\documentclass[12pt,letterpaper,final,twoside,leqno]{amsart}
\usepackage{eucal, mathrsfs}
\usepackage{amsmath,amssymb,amsthm,
amscd,amsopn}
\usepackage{hyperref}
\usepackage{times}
\usepackage{xspace}
\usepackage{tikz}
\usepackage{epsfig,epic,eepic,latexsym,color}
\usepackage[all]{xy}
\usepackage{paralist}
\usepackage{stmaryrd}
\usepackage{enumitem}
\usepackage{color}
\usepackage{soul}
\usepackage{mathtools}
\usepackage{rotating}

\catcode`~=11 \def\UrlSpecials{\do\~{\kern -.15em\lower .7ex\hbox{~}\kern .04em}} \catcode`~=13

\newcommand{\urlwofont}[1]{\urlstyle{same}\url{#1}}

\newcommand{\widepagestyle}{
\voffset=0in
\hoffset=0in
\marginparwidth=0.7in
\oddsidemargin=0in
\evensidemargin=0in
\textwidth=6.5in
\textheight=8.5in
\topmargin=0in
\headheight=0in
\headsep=0.2in
\footskip=0.5in
}

\newcounter{are-there-sections}
\DeclareMathAlphabet{\smallchanc}{OT1}{pzc}%
                                 {m}{it}
\DeclareFontFamily{OT1}{pzc}{}
\DeclareFontShape{OT1}{pzc}{m}{it}%
             {<-> s * [1.100] pzcmi7t}{}
\DeclareMathAlphabet{\mathchanc}{OT1}{pzc}%
                                 {m}{it}

\newcommand{\mcH}{\mathchanc{H}}

\newcommand{\mcm}{\mathchanc{m}}

\newcommand{\mco}{\mathchanc{o}}

\newcommand{\osX}{\overline{\sX}}

\newcommand{\tf}{\tilde{f}}
\newcommand{\tg}{\tilde{g}}

\DeclareFontFamily{OMS}{rsfs}{\skewchar\font'60}
\DeclareFontShape{OMS}{rsfs}{m}{n}{<-5>rsfs5 <5-7>rsfs7 <7->rsfs10 }{}
\DeclareSymbolFont{rsfs}{OMS}{rsfs}{m}{n}
\DeclareSymbolFontAlphabet{\scr}{rsfs}

\newcommand{\sA}{\scr{A}}
\newcommand{\sB}{\scr{B}}
\newcommand{\sC}{\scr{C}}

\newcommand{\sE}{\scr{E}}
\newcommand{\sF}{\scr{F}}
\newcommand{\sG}{\scr{G}}

\newcommand{\sI}{\scr{I}}

\newcommand{\sK}{\scr{K}}
\newcommand{\sL}{\scr{L}}

\newcommand{\sO}{\scr{O}}

\newcommand{\sT}{\scr{T}}

\newcommand{\sX}{\scr{X}}
\newcommand{\sY}{\scr{Y}}
\newcommand{\sZ}{\scr{Z}}

\newcommand{\bG}{\mathbb{G}}

\newcommand{\bL}{\mathbb{L}}

\newcommand{\bQ}{\mathbb{Q}}

\newcommand{\bZ}{\mathbb{Z}}

\newcommand{\fF}{\mathfrak{F}}

\newcommand{\fM}{\mathfrak{M}}

\newcommand{\fT}{\mathfrak{T}}

\newcommand{\fX}{\mathfrak{X}}
\newcommand{\fY}{\mathfrak{Y}}

\newcommand{\ofM}{\overline{\mathfrak{M}}}
\newcommand{\ofU}{\overline{\mathfrak{U}}}

\newcommand{\ff}{\mathfrak{f}}

\newcommand{\ug}{\underline{g}}
\newcommand{\uh}{\underline{h}}
\newcommand{\um}{\underline{m}}

\newcommand{\uX}{\underline{X}}

\newcommand{\usX}{\underline{\sX}}

\newcommand{\ol}{\overline}

\newcommand{\Def}{\mathfrak Def}

\newcommand{\Sch}{\mathfrak Sch}

\DeclareMathOperator{\Art}{{Art}}

\DeclareMathOperator{\qc}{{qc}}

\DeclareMathOperator{\Diff}{Diff}

\DeclareMathOperator{\codim}{codim}

\DeclareMathOperator{\Ker}{{Ker}}

\DeclareMathOperator{\depth}{{depth}}

\DeclareMathOperator{\Hom}{Hom}

\newcommand{\sHom}[0]{{\mcH\mco\mcm}}

\DeclareMathOperator{\im}{{im}}

\DeclareMathOperator{\Isom}{Isom}

\DeclareMathOperator{\Gor}{{Gor}}
 
\DeclareMathOperator{\Proj}{{Proj}}

\DeclareMathOperator{\red}{red}

\DeclareMathOperator{\rk}{{rk}}

\DeclareMathOperator{\Spec}{{Spec}}
\DeclareMathOperator{\supp}{{supp}}
\DeclareMathOperator{\Supp}{{Supp}}

\newcommand{\factor}[2]{\left. \raise 2pt\hbox{\ensuremath{#1}} \right/
        \hskip -2pt\raise -2pt\hbox{\ensuremath{#2}}}

\def\coh#1.#2.#3.{H^{#1}(#2,#3)}
\def\dimcoh#1.#2.#3.{h^{#1}(#2,#3)}
\def\hypcoh#1.#2.#3.{\mathbb H_{\vphantom{l}}^{#1}(#2,#3)}
\def\loccoh#1.#2.#3.#4.{H^{#1}_{#2}(#3,#4)}
\def\dimloccoh#1.#2.#3.#4.{h^{#1}_{#2}(#3,#4)}
\def\lochypcoh#1.#2.#3.#4.{\mathbb H^{#1}_{#2}(#3,#4)}
\def\ses#1.#2.#3.{0  \longrightarrow  #1   \longrightarrow 
 #2 \longrightarrow #3 \longrightarrow 0} 
\def\sesshort#1.#2.#3.{0
 \rightarrow #1 \rightarrow #2 \rightarrow #3 \rightarrow 0}
\def\dist#1.#2.#3.{  #1   \longrightarrow 
 #2 \longrightarrow #3 \stackrel{+1}{\longrightarrow} } 
\def\CDdist#1.#2.#3.{  #1   @>>>  #2  @>>>   #3 @>+1>> }  
\def\shortses#1.#2.#3.{0  \rightarrow  #1   \rightarrow 
 #2  \rightarrow   #3 \rightarrow  0}
\def\shortdist#1.#2.#3.{  #1   \rightarrow 
 #2  \rightarrow   #3 \stackrel{+1}{\rightarrow} }  
\def\ddist#1.#2.#3.#4.#5.#6.{\CD
#1 @>>> #2 @>>> #3 @>+1>> \\
@VVV @VVV @VVV \\
#4 @>>> #5 @>>> #6 @>+1>> 
\endCD}
\def\ddistun#1.#2.#3.#4.#5.#6.{\CD
#1 @>>> #2 @>>> #3 @>+1>> \\
@. @VVV @VVV  \\
#4 @>>> #5 @>>> #6 @>+1>> 
\endCD}
\def\Iff#1#2#3{
\hfil\hbox{\hsize =#1
\vtop{\noin #2}
\hskip.5cm 
\lower.5\baselineskip\hbox{$\Leftrightarrow$}\hskip.5cm
\vtop{\noin #3}}\hfil\medskip}
\newcommand{\union}\cup
\newcommand{\intersect}\cap
\newcommand{\Union}\bigcup
\newcommand{\Intersect}\bigcap
\def\myoplus#1.#2.{\underset #1 \to {\overset #2 \to \oplus}}

\newcommand{\of}{\overline{f}}

\newcommand{\ox}{\overline{x}}
\newcommand{\oy}{\overline{y}}

\newcommand{\oD}{\overline{D}}

\newcommand{\oF}{\overline{F}}

\newcommand{\oS}{\overline{S}}

\newcommand{\oX}{\overline{X}}
\newcommand{\oY}{\overline{Y}}

\makeatletter

\renewcommand\subsection{
  \renewcommand{\sfdefault}{pag}
  \@startsection{subsection}%
  {2}{0pt}{-\baselineskip}{.2\baselineskip}{\raggedright
    \sffamily\itshape\small\bfseries
  }}
\renewcommand\section{
  \renewcommand{\sfdefault}{phv}
  \@startsection{section} %
  {1}{0pt}{\baselineskip}{.2\baselineskip}{\centering
    \sffamily
    \scshape
    \bfseries
}}
\newcounter{lastyear}
\addtocounter{lastyear}{-1}

\newcommand\noin{\noindent}

\newtheoremstyle{bozont}{8pt}{10pt}%
     {\itshape}
     {}
     {\bfseries}
     {.}
     {.5em}
     {\thmname{#1}\thmnumber{ #2}\thmnote{ \rm #3}}
\newtheoremstyle{bozont-sf}{3pt}{3pt}%
     {\itshape}
     {}
     {\sffamily}
     {.}
     {.5em}
     {\thmname{#1}\thmnumber{ #2}\thmnote{ \rm #3}}
\newtheoremstyle{bozont-sc}{3pt}{3pt}%
     {\itshape}
     {}
     {\scshape}
     {.}
     {.5em}
     {\thmname{#1}\thmnumber{ #2}\thmnote{ \rm #3}}
\newtheoremstyle{bozont-remark}{8pt}{15pt}%
     {}
     {}
     {\scshape}
     {}
     {.5em}
     {\thmname{#1}\thmnumber{ #2.}\thmnote{\ \  ( #3)}}
\newtheoremstyle{bozont-def}{8pt}{15pt}%
     {}
     {}
     {\bfseries}
     {.}
     {.5em}
     {\thmname{#1}\thmnumber{ #2}\thmnote{ \rm #3}}
\newtheoremstyle{bozont-reverse}{3pt}{3pt}%
     {\itshape}
     {}
     {\bfseries}
     {.}
     {.5em}
     {\thmnumber{#2.}\thmname{ #1}\thmnote{ \rm #3}}
\newtheoremstyle{bozont-reverse-sc}{3pt}{3pt}%
     {\itshape}
     {}
     {\scshape}
     {.}
     {.5em}
     {\thmnumber{#2.}\thmname{ #1}\thmnote{ \rm #3}}
\newtheoremstyle{bozont-reverse-sf}{3pt}{3pt}%
     {\itshape}
     {}
     {\sffamily}
     {.}
     {.5em}
     {\thmnumber{#2.}\thmname{ #1}\thmnote{ \rm #3}}
\newtheoremstyle{bozont-remark-reverse}{3pt}{3pt}%
     {}
     {}
     {\sc}
     {.}
     {.5em}
     {\thmnumber{#2.}\thmname{ #1}\thmnote{ \rm #3}}
\newtheoremstyle{bozont-def-reverse}{3pt}{3pt}%
     {}
     {}
     {\bfseries}
     {.}
     {.5em}
     {\thmnumber{#2.}\thmname{ #1}\thmnote{ \rm #3}}
\newtheoremstyle{bozont-def-newnum-reverse}{3pt}{3pt}%
     {}
     {}
     {\bfseries}
     {}
     {.5em}
     {\thmnumber{#2.}\thmname{ #1}\thmnote{ \rm #3}}
\theoremstyle{bozont}    
\ifnum \value{are-there-sections}=0 {%
  \newtheorem{proclaim}{Theorem}
} 
\else {%
  \newtheorem{proclaim}{Theorem}[section]
} 
\fi
\newtheorem{thm}[proclaim]{Theorem}

\newtheorem{cor}[proclaim]{Corollary} 
 
\newtheorem{lem}[proclaim]{Lemma} 
\newtheorem{prop}[proclaim]{Proposition}

\theoremstyle{bozont-sc}
\newtheorem{proclaim-special}[proclaim]{\specialthmname}

\theoremstyle{bozont-remark}

\newtheorem{rem}[proclaim]{Remark}

\newtheorem*{SubHeading*}{\SubHeadingName}%
\newtheorem{SubHeading}[proclaim]{\SubHeadingName}
\newtheorem{sSubHeading}[equation]{\sSubHeadingName}
\newenvironment{demo-r}[1]{\def\SubHeadingName{#1}\begin{SubHeading-r}}
  {\end{SubHeading-r}}%
\newenvironment{subdemo-r}[1]{\def\sSubHeadingName{#1}\begin{sSubHeading-r}}
  {\end{sSubHeading-r}} %
\newenvironment{demo*}[1]{\def\SubHeadingName{#1}\begin{SubHeading*}}
  {\end{SubHeading*}}%

\newtheorem{defn-thm}[proclaim]{Definition--Theorem}  

\theoremstyle{bozont-def}    
\newtheorem{defn}[proclaim]{Definition}
\newtheorem{notation}[proclaim]{Notation} 

\theoremstyle{bozont-reverse}    

\theoremstyle{bozont-reverse-sc}
\newtheorem{proclaimr-special}[proclaim]{\specialthmname}
{\def\specialthmname{#1}\begin{proclaimr-special}}%
{\end{proclaimr-special}}
\theoremstyle{bozont-remark-reverse}

\newtheorem{SubHeading-r}[proclaim]{\SubHeadingName}
\newtheorem{sSubHeading-r}[equation]{\sSubHeadingName}
\newtheorem{SubHeadingr}[proclaim]{\SubHeadingName}

\theoremstyle{bozont-def-newnum-reverse}    

\theoremstyle{bozont-def-reverse}

\newtheorem{newnumspecial}[proclaim]{\specialnewnumname}

\numberwithin{equation}{proclaim}

\numberwithin{figure}{section}

\newenvironment{enumerate-p}{
  \begin{enumerate}}
  {\setcounter{equation}{\value{enumi}}\end{enumerate}}
\newenvironment{enumerate-cont}{
  \begin{enumerate}
    {\setcounter{enumi}{\value{equation}}}}
  {\setcounter{equation}{\value{enumi}}
  \end{enumerate}}

\newlength{\swidth}
\makeatother

\widepagestyle

\address{Zsolt Patakfalvi, Princeton University, Department of Mathematics, Fine Hall, Washington Road,
Princeton, NJ-08544-1000, USA
}
\email{pzs@math.princeton.edu}
\urladdr{http://www.math.princeton.edu/\~{}pzs}

\newtheorem*{thm_vanishing_log_canonical}{Theorem \ref{thm:vanishing_log_canonical}}

\title{Fibered stable varieties}
\author{ Zsolt Patakfalvi}

\begin{document}

\begin{abstract}
We show that if a stable variety (in the sense of Koll\'ar and Shepherd-Barron) admits a fibration with stable fibers and base, then this fibration structure deforms (uniquely) for all small deformations. During our proof we obtain  a Bogomolov-Sommese type vanishing for vector bundles and reflexive differential $n-1$-forms as well. 
\end{abstract}

\maketitle

\tableofcontents

\section{Introduction}
\label{sec:introduction}

The moduli space $\ofM_h$ of stable varieties (or equivalently of semi-log canonical models) with Hilbert polynomial $h$  is the natural generalization of the widely investigated space $\ofM_g$ of stable curves of genus $g$ \cite{Kollar_Moduli_of_varieties_of_general_type}, \cite{Kollar_Shepher_Barron_Threefolds_and_deformations}, \cite{Kollar_Projectivity_of_complete_moduli}. It parametrizes (possibly reducible) varieties with semi-log canonical singularities and ample canonical bundle. In \cite{Bhatt_Ho_Patakfalvi_Schnell_Moduli_of_products_of_stable_varieties} connected components containing products of stable varieties were described very precisely. It turned out that if a stable variety admits a product structure, then so do all its deformations. Instead of having a product structure, one can look at the weaker condition: having a fibration structure with stable fibers and base.  Then the fibration structure does not extend to all deformations as a product structure, because of certain monodromy issues 
in the limit at infinity \cite{Abramovich_Vistoli_Compactifying_the_space_of_stable_maps}. However, according to the main result of the paper, the fibration structure does extend to small deformations. 


\begin{thm}
\label{thm:vague}
Let $k$ be an algebraically closed field of characteristic zero. If a stable variety $X$ admits a fibration structure $f \colon X \to Y$  with stable fibers and base, then
\begin{enumerate}
\item \label{itm:vague:non_Q_Gorenstein} For every deformation $X'$ of $X$ over an Artinian local $k$-algebra $A$ there is a unique deformation of $f \colon X \to Y$ over $A$ of the form $f' \colon X' \to Y'$ such that $Y'_A \cong Y$ and $f'_A = f$ via this isomorphism.
\item  \label{itm:vague:Q_Gorenstein} If further the above deformation $X' \to \Spec A$ is a stable deformation then  both $f'$ and $Y' \to \Spec A$ are stable families. Here stable family means that it also satisfies Koll\'ar's condition, that is, the reflexive powers of the relative canonical sheaves commute with base change.
\end{enumerate}
\end{thm}

Point \eqref{itm:vague:Q_Gorenstein} of Theorem \ref{thm:vague} is equivalent to the following moduli theoretic statement. To state it, fix three $\bZ \to \bZ$ functions  $h_1$, $h_2$ and $h$, and consider the following pseudo-functor $\fF \fM_{(h_1, h_2), h}$: for a test $k$-scheme $B$, $\fF \fM_{(h_1, h_2), h}(B)$  consists of fibrations $ X \to Y \to B$, where both $Y \to B$ and $X \to Y$ are stable families and  the Hilbert functions of $Y \to B$, $X \to Y$ and of $X \to B$ are $h_1$, $h_2$ and $h$, respectively.   We will prove that $\fF \fM_{(h_1,h_2), h}$ is a DM-stack locally of finite type over $k$.  Furthermore, there is a forgetful map $F \colon \fF \fM_{(h_1,h_2), h} \to \ofM_h$ obtained by forgetting the fibration structure of $X$. Then the equivalent rewording of point \ref{itm:vague:Q_Gorenstein} of Theorem \ref{thm:vague} is:

\begin{thm}
\label{thm:main}
The forgetful morphism $F \colon \fF \fM_{(h_1,h_2), h} \to \ofM_h$ is \'etale.
\end{thm}

An immediate consequence of Theorem \ref{thm:main} is as follows.

\begin{cor}
\label{cor:dense_image}
 The image of $F \colon \fF \fM_{(h_1,h_2), h} \to \ofM_h$ is dense in every component it
intersects. 
\end{cor}

In the special cases, when $\deg h_1 = \deg h_2 = 1$, i.e., when the fibers and the base of $f$ are curves, a compactification  of each connected component of $\fF  
\fM_{(h_1, h_2), h}$ can be obtained as $ \sK_{g_1,0}^{\mathrm{bal}}(\ofM_{g_2},d) $ for adequate values of $g_1$, $g_2$ and $d$. Here $\ofM_{g_2}$ is the usual space of stable curves with genus $g_2$ and $\sK_{g_1, 0}^{\mathrm{bal}} ( \_, d)$ is the
 Abramovich-Vistoli space of stable maps  \cite{Abramovich_Vistoli_Compactifying_the_space_of_stable_maps}. Note also that $g_1$ and $g_2$ are just the genera given by $h_1$ and $h_2$. On the other hand $d$ depends on the actual connected component of $\fF \fM_{(h_1, h_2), h}$ considered. One can show now  that $F$ extends naturally to a forgetful morphism
$\oF \colon \sK_{g_1,0}^{\mathrm{bal}}(\ofM_{g_2},d) \to \ofM_{h}$ \cite[Notation 7.2]{Patakfalvi_Arakelov_Parshin_rigidity_of_towers_of_curve_fibrations}. Since every component of both $\sK_{g_1,0}^{\mathrm{bal}}(\ofM_{g_2},d)$ and $\ofM_h$ is proper, and the image of $\oF$ is dense in the relevant components according to Corollary \ref{cor:dense_image}, $\oF$ is surjective onto every irreducible component that it intersects. Therefore, the one parameter degenerations of stable surfaces admitting a stable fibration structure are coarse moduli spaces of stacks admitting a twisted stable fibration structure in the sense of Abramovich-Vistoli. 
Note that these observations are crucial for the results of \cite{Patakfalvi_Arakelov_Parshin_rigidity_of_towers_of_curve_fibrations}. To generalize these considerations to higher dimensions it would be necessary to  generalize the Abramovich-Vistoli construction itself to higher dimensions \cite{Abramovich_Vistoli_Compactifying_the_space_of_stable_maps}. Note that Alexeev defined \emph{non-twisted} stable maps from surfaces in \cite{Alexeev_Moduli_spaces_M_g_n_W_for_surfaces}. It would be interesting to extend that to the stack target and possibly to arbitrary dimensional source case.


Note that questions similar to Theorem \ref{thm:vague} have been considered by Catanese, e.g.,  \cite{Catanese_Moduli_and_classification_of_irregular_Kahler_manifolds_with_Albanese_general_type_fibrations,Catanese_Fibered_surfaces_varieties_isogenous_to_a_product_and_related_moduli_spaces}. In  \cite{Catanese_Moduli_and_classification_of_irregular_Kahler_manifolds_with_Albanese_general_type_fibrations} it is shown that fibration structures $f \colon X \to Y$ extend to small deformations if $X$ is smooth, projective and $Y$ is a smooth curve of genus at least two (or generally a variety of maximal Albanese dimension). This is in fact stronger statement than ours in the  $\dim Y=1$ case, since $f$ is allowed to have arbitrarily bad special fibers. One of the main reason for this difference is that the methods of \cite{Catanese_Moduli_and_classification_of_irregular_Kahler_manifolds_with_Albanese_general_type_fibrations} are topological: it is shown that a fibration structure as above is a 
topological 
property. On the other hand our methods are purely deformation theoretic. In particular, our methods not only yield that every nearby variety has a similar fibration structure, but also that for families the fibration structure extends for the whole family after an \'etale base-change. 

Further note that similar deformation theoretic arguments as ours were used in \cite[Theorem 33]{Kollar_Deformations_of_elliptic_Calabi_Yau_manifolds}, which yields considerably more general statement when  $X$  has rational and  $Y$ has canonical singularities. 

%

%

During the proof of Theorem \ref{thm:main}, the following  two vanishing results are obtained, the first of which is implied by the second one. Note that Theorem \ref{thm:vanishing_log_canonical} is a vector bundle version of a  special case of the Bogomolov-Sommese vanishing for reflexive differentials \cite[Theorem 7.2]{Greb_Kebekus_Kovacs_Peternell_Differential_forms_on_log_canonical_spaces}. 

\begin{thm}
\label{thm:vanishing_hom_from_cotangent}
If  $X$ is a  stable variety, and $\sE$ a anti-nef vector bundle on $X$,
then $\Hom_X(\Omega_{X},  \sE)$ $=\Hom_X(\bL_{X},  \sE)=0$ (here $\bL_X$ is the cotangent complex of $X$). 
\end{thm}

\begin{thm}
\label{thm:vanishing_log_canonical}
If $X$ is a  projective variety of dimension $n$, $D \geq 0$ a $\bQ$-divisor on $X$
such that $(X,D)$ is log canonical, $\sL$ an anti-ample $\bQ$-line bundle, $\sE$ an
anti nef vector bundle, then
\begin{equation}
\label{eq:vanishing_log_canonical:vanishing}
 H^0(X, \Omega^{[n-1]}_X (\log \lfloor D \rfloor )  [\otimes] \sL \otimes
\sE ) = 0 .
\end{equation}
\end{thm}

Further, in the above statements  anti-nef  could be easily replaced by the more technical term of weakly-negative. We keep the anti-nef version to avoid unnecessary technicalities. It is also expected that most results of the paper hold in the log case as well, i.e., when  stable varieties are replaced by stable pairs. However, we made the decision to keep the log-free versions since the deformation theory part, i.e., Section \ref{sec:deformation}, would have been considerably longer in the log case. This is partially due to the fact that even the starting point of our deformation theory considerations (i.e., \cite{Abramovich_Hassett_Stable_varieties_with_a_twist}) uses the non-log setting. On the other hand,  Theorem \ref{thm:vanishing_log_canonical} is already phrased in the log-setting.

\subsection{Idea of the proof and organization}

Consider a fibration $f \colon X \to Y$ as in Theorem \ref{thm:vague}. According to \cite[Propositions 3.9 and 3.10]{Bhatt_Ho_Patakfalvi_Schnell_Moduli_of_products_of_stable_varieties}, to equate the (unconstrained) deformation theory of the fibration and of $X$, the most important step is to prove that $\Hom_{Y} ( \Omega_{Y}, R^1 f_* \sO_X)=0$  (see also \cite[Theorem 8.1 and Theorem 8.2]{Horikawa_On_deformations_of_holomorphic_maps_III}). One can obtain this from Proposition \ref{thm:semi_negative} and Theorem \ref{thm:vanishing_hom_from_cotangent}. However, there is a subtlety in the deformation theory of stable varieties which makes things considerably harder. The deformation theory of an object in $\ofM_h$ is not given by the unconstrained deformation theory of the corresponding stable variety, but the deformation theory of its index-one covering stack \cite{Abramovich_Hassett_Stable_varieties_with_a_twist}. This index-one covering stack is a finite, 
birational cover whose canonical bundle  is a line bundle. Therefore, one has to pass to index-one 
covers and then apply 
\cite[Propositions 3.9 and 3.10]{Bhatt_Ho_Patakfalvi_Schnell_Moduli_of_products_of_stable_varieties}. There is a further subtlety  at this point. The natural cover to consider for the  deformation theory of $f \colon X \to Y$ as a stable family is the relative index-one-cover of $X$ over the index-one-cover of $Y$. First, this is somewhat hard to deal with because the base of this relative index-one-cover is a stack. So, equating the deformation theory of it to the stable deformation theory of $f \colon X \to Y$ is considerably longer than the absolute case in \cite{Abramovich_Hassett_Stable_varieties_with_a_twist}. Second, this relative index-one-cover does not agree with the absolute index-one-cover of $X$. On the other hand the former does map to the later and one can prove that their deformation theories are the same via this map. However, this yields another layer of extra technicalities to the discussion. 

The passage to index-one covers is worked out in Section \ref{sec:deformation}. Section \ref{sec:definition} contains the precise definition of the objects of Theorems \ref{thm:main}. Sections \ref{sec:vanishing} is devoted to the proofs of the above mentioned Proposition \ref{thm:semi_negative} and Theorem \ref{thm:vanishing_hom_from_cotangent}, while the proof of Theorem \ref{thm:main} is finished in Section \ref{sec:main_theorem}.

\subsection{Acknowledgement}


The author is thankful to Fabrizio Catanese and J\'anos Koll\'ar for the useful remarks. 

\subsection{Notation}
\label{subsec:notations}

We work over an algebraically closed field $k$ of characteristic zero. All schemes and stacks are noetherian and separated over $k$. 
A noetherian scheme $X$ is \emph{relatively $S_d$} over $B$, if $X_b$ is $S_d$ for every $b \in B$. In the same situation if $X_b$ is Gorenstein in codimension one for all $b \in B$, then $X$ is \emph{relatively $G_1$} over $B$. The absolute version of these and of all the other following notions is obtained by simply taking $B= \Spec k$. Since depth of a point and being Gorenstein are formal local properties, being $S_d$ or Gorenstein can be defined for DM-stacks by requiring them on \'etale covers by schemes. Then the above notions do make sense for DM-stacks.

For an arbitrary coherent sheaf $\sF$ on a scheme $X$, the \emph{reflexive hull} of $\sF$ is $\sF^{**}$. \emph{Reflexive power, pullback, tensor product, etc} is defined by taking power, pullback, tensor product, etc and then reflexive hull. E.g., the second reflexive power $\sF^{[2]}$ is $\left( \sF^{\otimes 2}\right)^{**}$. Reflexive operations are denoted by putting square brackets around the usual operation signs. E.g., reflexive pullback is denoted by $f^{[*]}$ and reflexive tensor product by $[\otimes ]$. Reflexive (log-)differentials are denoted by $\Omega_X^{[i]}(\log D)$ and coherently with the above discussion are $(\Omega_X^i(\log D))^{**}$.
 Let $X$ be flat and relatively $S_2$, $G_1$ over $B$. The sheaf $\sF$ on $X$   is a $\bQ$-line bundle, if it is reflexive, a line bundle in relative codimension one, and  $\sF^{[m]}$ is a line bundle for some $m \neq 0$. In particular, by \cite[Proposition 3.6]{Hassett_Kovacs_Reflexive_pull_backs} then $\sF^{[im]} \cong (\sF^{[m]})^i$. A $\bQ$-line bundle is nef, relatively ample, etc. if $\sF^{[m]}$ is nef for any $m$ such that $\sF^{[m]}$ is a line bundle. By the discussion above, this definition does make sense.

\emph{Vector bundle} means a locally free sheaf of finite rank. \emph{Line bundle} means a locally free sheaf of rank one. When it does not cause any misunderstanding, pullback is denoted by lower index. E.g., if $\sF$ is a sheaf on $X$, and $X \to Y$ and $Z \to Y$ are morphisms, then $\sF_Z$ is the pullback of $\sF$ to $X \times_Y Z$. This unfortunately is also a source of some confusion: $\sF_y$ can mean both the stalk and the fiber of the sheaf $\sF$ at the point $y$. Since both are frequently used notation in the literature, we opt to use both and hope that it will always be clear from the context which one we mean.

A \emph{representable} morphism of stacks means representable by schemes. A proper DM-stack with a coarse moduli space is projective if and only if so is its coarse moduli space. A $\bQ$-line bundle or a $\bQ$-Cartier divisor $L$ on a DM-stack $\sX$ is (relatively) ample, if the descent of a high enough multiple of $L$ to the coarse moduli space (given that that exits) is (relatively) ample. This is equivalent to saying that for any finite cover $Y$ of $\sX$ by a scheme the pullback of $L$ to $Y$ is (relatively) ample. Note that this definition really works in the relative case only if the base is a scheme. If it is a stack, then we pull back our family via an \'etale cover of the base, and we apply the above definition there.
Since taking coarse moduli space commutes with base change \cite[Lemma 2.3.3]{Abramovich_Vistoli_Compactifying_the_space_of_stable_maps}, if $\sX$ is projective over the base, then $L$ is relatively ample if and only if it is  ample over every fiber over ever $k$-point of the base (this works even if the base is a DM-stack as well) \cite[Theorem 1.7.8]{Lazarsfeld_Positivity_in_algebraic_geometry_I}. The category $\Sch_k$ is the category of schemes over $k$. Square brackets around quotients, e.g., $[P/G]$, means stack quotient.

All derived category computations of the article take place in $D_{\qc}(X)$, the derived category of unbounded complexes with quasi-coherent cohomology  sheaves. In our situation this is equivalent to the derived category of complexes of quasi-coherent modules via the natural embedding of the latter into $D_{\qc}(X)$. Furthermore the derived functors behave compatibly with this equivalence \cite[page 207]{Neeman_The_Grothendieck_duality_theorem}. Also, the  usual bounded derived categories are full subcategories of $D_{\qc}(X)$, again with agreeing derived functors. We need to use the unbounded derived category, because the cotangent complex 
$\bL_X$ of a scheme (or DM-stack) is unbounded (from below). If $\sC \in D_{\qc}(X)$, then $h^i(\sC)$ is the $i$-th cohomology sheaf and $H^i(X,\sC)$ is the $i$-th hypercohomology of $\sC$. If $f \colon X \to Y$ is a morphism, $R^{< i}f_* \sC$ and $R^{\leq i }f_* \sC$ mean the adequate truncations of $R f_* \sC$. 

The abbreviations \emph{lc} and \emph{slc} mean log canonical and semi-log canonical, respectively. If $S$ is a reduced divisor on a (demi-)normal scheme, $0 \leq \Delta$ a $\bQ$-divisor, and $S$ a reduced divisor with normalization $S^n$, then  $\Diff_S \Delta$ and $\Diff_{S^n} \Delta$ denote the different \cite[Different 4.2]{Kollar_Singularities_of_the_minimal_model_program}.

\section{Definition of the moduli spaces and forgetful maps}
\label{sec:definition}

In this section we define precisely the moduli space $\fF \fM_{(h_1,h_2),h}$, and then after some technical preparation we define the functor $F \colon \fF \fM_{(h_1, h_2),h} \to \ofM_h$ of Theorem \ref{thm:main}. 

\subsection{The moduli spaces}

First, shortly we recall the definition of stable varieties, and define the moduli space $\fF \fM_{(h_1,h_2),h}$.

\begin{defn}
A noetherian scheme is \emph{demi-normal}, if it is $S_2$ and nodal in codimension one \cite[Definition 5.1]{Kollar_Singularities_of_the_minimal_model_program}. Here nodal is meant in the sense of \cite[1.41]{Kollar_Singularities_of_the_minimal_model_program}.
\end{defn}

\begin{defn}
\label{defn:slc}
Let $X$ be a demi-normal scheme and $\pi \colon \oX \to X$ its normalization. Then the (reduced) double locus of $\pi$ on $\oX$ is of pure codimension $1$ and is called the \emph{conductor} of $X$. Denote it by $\oD$. The scheme $X$ is \emph{semi-log canonical} (or shortly \emph{slc}), if $K_X$ is $\bQ$-Cartier and $(\oX, \oD)$ is log canonical \cite[Definition-Lemma 5.10]{Kollar_Singularities_of_the_minimal_model_program}. 

If there is also a $\bQ$-Weil divisor $\Delta$ given on $X$, which avoids the codimension one singular point of $X$, then we can also define when the pair $(X, \Delta)$ is slc. In this situation $\Delta$ is $\bQ$-Cartier in codimension one, and then $\ol{\Delta}:=\pi^* \Delta$ is defined as the unique extension of the pullback over the $\bQ$-Cartier locus of $\Delta$. Then, $(X, \Delta)$ is  slc, if $K_X + \Delta$ is $\bQ$-Cartier and $(\oX, \oD + \ol{\Delta})$ is log canonical (see \cite[Definition-Lemma 5.10]{Kollar_Singularities_of_the_minimal_model_program}, and note that it works for non-effective $\Delta$ as well).

\end{defn}

\begin{notation}
If $X$ is a demi-normal scheme, then saying that $\pi\colon (\oX , \oD) \to X$ is the normalization means that $\oD$ is the conductor divisor, on the normalization $\oX$ of $X$. The (reduced) divisor of the double locus on $X$, i.e., $(\pi_* \oD)_{\red}$, is also called the conductor, since it is defined by the same ideal $\sI \subseteq \sO_X$ as $\oD$ (this ideal lies a priori in $\pi_* \sO_{\oX}$, but it happens to be contained in $\sO_X$). 
\end{notation}


\begin{defn}
\label{defn:stable_variety}
A \emph{stable} variety is an equidimensional, connected, proper, slc scheme over a field, such that $\omega_X$ is ample. The function $h(m):=\chi \left( \omega_X^{[m]} \right)$ is called the \emph{Hilbert function} of $X$.
\end{defn}

\begin{defn}
\label{defn:stable_family}
A family of stable varieties is a flat morphism $f \colon X \to B$ satisfying \emph{Koll\'ar's condition}. That is, for all $m \in \bZ$ and $b \in B$, $X_b$ is a stable variety and $\omega_{X/B}^{[m]}$ is flat, and for every base change $\tau \colon B' \to B$ and the induced morphism $\rho\colon X_{B'} \to X$, the natural homomorphism
\begin{equation}
\label{eq:stable_family:Kollar_condition}
\rho^*  \left( \omega_{X/B}^{[m]} \right) \to \omega_{X_{B'}/B'}^{[m]}
\end{equation}
is an isomorphism.
\end{defn}

\begin{notation}
\label{notation:stable_moduli_space}
One may consider then the category of all  stable families with fixed Hilbert function $h$. One can show that this forms a proper DM-stack of finite type over $k$ \cite[Theorem 2.8]{Bhatt_Ho_Patakfalvi_Schnell_Moduli_of_products_of_stable_varieties}, and it is denoted by $\ofM_h$ in this article. The category of all stable families of relative dimension $m$ is denoted by $\ofM_m$. That is, $\ofM_m= \bigcup_{\deg h = m} \ofM_h$.
\end{notation}

\begin{defn}
\label{defn:fibration_of_stable_varieties}
A  \emph{fibration of stable varieties with Hilbert function vector $\uh=(h_1,h_2)$} over a base scheme $B$   is a commutative diagram
\begin{equation}
\label{eq:fibration_of_stable_schemes}
\xymatrix{
    X_2 = X  \ar@/^2pc/[rr]^f  \ar[r]_{f_2=g} & X_1 = Y \ar[r]_{f_1} &  X_0= B
} ,
\end{equation}
such that $f_i$ is a family of stable varieties (satisfying Koll\'ar's condition), and $\chi \left(  \omega_{(X_i)_y}^{[m]} \right) = h_i(m)$ for every $m \in \bZ$, $1 \leq i \leq 2$ and $y \in X_{i-1}$. Define the category fibered in groupoids $\fF \fM_{\uh}$ over $\Sch_k$ to have such fibrations as objects over $B$, and natural Cartesian pullbacks as morphisms. I.e. a morphism  of $(X \to Y \to B)$ to $(X' \to Y' \to B')$ is a diagram as follows with all the squares being Cartesian.
\begin{equation*}
\xymatrix{
X \ar[r] \ar[d] &  X' \ar[d] \\
Y \ar[r] \ar[d] &  Y' \ar[d] \\
B \ar[r]  &  B'  \\
}
\end{equation*}
Sometimes we further require the Hilbert function of $f$ to be a fixed polynomial $h$. We denote the category obtained that way by $\fF \fM_{(h_1,h_2),h}$. For a vector of integers $\um=(m_1,m_2)$ define also the category  of all fibrations with dimension vector $\um$ as follows (with the morphisms being only the ones induced from $\fF \fM_{\uh}$).
\begin{equation*}
\fF \fM_{\um}:= \bigcup_{\uh=(h_1,h_2), \deg h_i = m_i} \fF \fM_{\uh}
\end{equation*}
\end{defn}

\begin{notation}
\label{notation:tower_of_stable_schemes}
Given a fibration as in \eqref{eq:tower_of_stable_schemes}, we use the short notations $( X \to Y \to B)$, $\uX$ or $(X_i,f_i)$ for it. 
\end{notation}

\begin{prop}
\label{prop:stack_isomorphism}
Let $\ofM_n$ denote the moduli stack of all stable varieties of dimension $n$, and $\ofU_n$ the universal family over it. Then, 
\begin{equation}
\label{eq:stack_isomorphism:statement}
\fF \fM_{(m_1,m_2)} \cong \underline{\Hom}_{\ofM_{m_1}}( \ofU_{m_1}, \ofM_{m_2} \times \ofM_{m_1}) 
\end{equation}
and hence it is a DM-stack locally of finite type.
\end{prop}

\begin{proof}
 There is a forgetful map $\pi \colon \fF \fM_{(m_1,m_2)} \to \ofM_{m_1}$ remembering only $Y$ of a fibration in \eqref{eq:fibration_of_stable_schemes}. We prove \eqref{eq:stack_isomorphism:statement}, by showing an isomorphism over $\ofM_{m_1}$, using $\pi$ as the structure map on the left and the natural projection on the right. So, fix $[Y \to B] \in \ofM_{m_1}$.  Given an element of $( X \to Y \to B) \in \fF \fM_{(m_1, m_2)}$ over $[Y \to B]$,  yields a family of stable varieties over $Y$ of dimension $m_2$.
Hence, $\uX=(X \to Y \to B)$ defines a morphism $\nu_{\uX} \colon Y \to \ff \fM_{(m_1,m_2)}$. Furthermore, since $\fF \fM_{(m_1, m_2)}$ represents the moduli problem of fibrations with dimension vector $(m_1,m_2)$, automorphisms of $\uX$ over $[Y \to B]$ and automorphisms of $\nu_{\uX}$ also match up. Hence we obtain the following string of isomorphisms of groupoids.
\begin{multline}
\label{eq:stack_isomorphism:long}
\fF \fM_{(m_1, m_2)}([Y \to B]) \cong  \Hom( Y, \ofM_{m_2}) \cong \Hom_{B}( Y, \ofM_{m_2} \times B) \\ \cong \Hom_{B}( (\ofU_{m_1})_B, \ofM_{m_2} \times B) =_{\mathrm{def}}  \underline{\Hom}_{\ofM_{m_1}}( \ofU_{m_1}, \ofM_{m_2} \times \ofM_{m_1})([Y \to B]) ,
\end{multline}
where
\begin{itemize}
\item $\Hom$ means the groupoid of functors over the base space 
\item $\underline{\Hom}$ is the $\Hom$-stack \cite{Olsson_Home_stacks_and_restriction_of_scalars} and
\item putting $([Y \to B])$ after a category means the fiber over $[Y \to B]$. 
\end{itemize}
The isomorphisms of \eqref{eq:stack_isomorphism:long} are all natural with respect to Cartesian maps 
\begin{equation*}
\xymatrix{
Y' \ar[r] \ar[d] & Y \ar[d] \\
B' \ar[r] & B .
}
\end{equation*}
Hence, \eqref{eq:stack_isomorphism:long} really yields an isomorphism as in \eqref{eq:stack_isomorphism:statement} over $\ofM_{m_1}$.
 
To prove the DM-stack statement it is enough to show that $\underline{\Hom}_{\fM}(\fX,\fY)$ is a DM-stack locally of finite type over $k$ whenever $\fM$, $\fX$ and $\fY$ are DM-stacks locally of finte type over $k$ and $\fX$ is proper, flat and representable over $\fM$. This is shown in \cite[Theorem 1.1]{Olsson_Home_stacks_and_restriction_of_scalars}, when $\fM$ is an algebraic space. To deduce it for a DM-stack $\fM$, one replaces $\fM$ with one of its \'etale atlases. This finishes our proof.
\end{proof}

\subsection{Adjunction}
\label{sec:adjunction}

Having defined the moduli spaces of Theorem \ref{thm:main}, the last goal of Section \ref{sec:definition} is to define the forgetful morphism $F \colon \fF \fM_{\um} \to \ofM_m$ of Theorem \ref{thm:main}. 
We need to show that the composite morphism $f$ of \eqref{eq:fibration_of_stable_schemes} is a family of stable varieties. In particular this involves showing that the total space of a family of slc varieties over an slc base is slc. The technical tool for this is inversion of adjunction, which relates the singularities of a divisor to the singularities of the total space close to the divisor. Since we are not aware of a good reference of inversion of adjunction for reducible total spaces,  we include it here.

For inductional reasons we need to use at certain places slc pairs, not only varieties.

\begin{lem}
\label{lem:inversion_of_adjunction}
Let $X$ be a demi-normal scheme, $\oD$ its conductor divisor, $S$ a reduced 
divisor with normalization $S^n \to S$ and $\Delta \geq 0$ a $\bQ$-divisor such that no
two of $\oD$, $S$ and $\Delta$ have common components and $K_X + S + \Delta$ is
$\bQ$-Cartier. In this case, $(X,S + \Delta)$ is slc near $S$ if and only if $(S^n,
\Diff_{S^n} (\Delta) )$ is lc.
\end{lem}

\begin{proof}
Let $\pi \colon (\oX,\oD + \oS +  \ol{\Delta}) \to (X, S+ \Delta)$ be the normalization (i.e., $\oD$ is the conductor) of $X$, $n\colon \oD^n \to \oD$ the normalization of the conductor and $\tau \colon \oD^n \to \oD^n$ the order two automorphism exchanging the preimages of the nodes \cite[5.2]{Kollar_Singularities_of_the_minimal_model_program}. Note then that by the arguments of \cite[5.7]{Kollar_Singularities_of_the_minimal_model_program}
\begin{equation}
\label{eq:inversion_of_adjunction:two}
 \Diff_{S^n}(\Delta) = \Diff_{S^n}(\oD + \ol{\Delta}).
\end{equation}
First, assume that $(X, S + \Delta)$ is slc near $S$. Then by \cite[Definition-Lemma 5.10]{Kollar_Singularities_of_the_minimal_model_program}, $(\oX, \oD  + \oS  + \ol{\Delta})$ is lc near $\oS$. Therefore by adjunction $(S^n,
\Diff_{S^n} (\oD + \ol{\Delta}) )$ is lc. Finally using \eqref{eq:inversion_of_adjunction:two} yields  the forward direction of the lemma.

To prove the backwards direction, assume that $(S^n,
\Diff_{S^n} (\Delta) )$ is lc. Then by \eqref{eq:inversion_of_adjunction:two}, so is $(S^n,
\Diff_{S^n} (\oD + \ol{\Delta)} )$. Hence by inversion of adjunction on normal varieties \cite[Theorem]{Kawakita_Inversion_of_adjunction_on_log_canonicity}, $(\oX, \oD +\oS + \ol{\Delta})$ is lc around $\oS$. \emph{We claim that in fact it is lc around $\pi^{-1} \pi(\oS)$.} This will yield the backwards direction by again using \cite[Definition-Lemma 5.10]{Kollar_Singularities_of_the_minimal_model_program}. 

To prove the claim, fix a point $P \in \pi^{-1} \pi(\oS)$. If $P \not\in D$, then the claim is immediate. Hence we assume $P \in D$. By inversion of adjunction on normal varieties \cite[Theorem]{Kawakita_Inversion_of_adjunction_on_log_canonicity}, $(X, \oD + \oS + \ol{\Delta})$ is lc around $P$ if and only if so is $(\oD^n, \Diff_{\oD^n}( \oS + \ol{\Delta} ))$ at points over $P$. However, by \cite[(5.3), proof of (5.33)]{Kollar_Singularities_of_the_minimal_model_program} and the fact that $P \in \pi^{-1} \pi(\oS)$, for every point $Q$ over $P$, there is a finite sequence $Q=Q_1,\dots,Q_r$, such that $\tau(Q_i) = Q_{i+1}$ and $n(Q_r) \in \oS$. Further, $(\oD^n, \Diff_{\oD^n}( \oS + \ol{\Delta} ))$ is $\tau$-invariant by \cite[Proposition 5.12]{Kollar_Singularities_of_the_minimal_model_program}. Hence by downward induction on $i$, using that $\tau$ is an isomorphism, $(\oD^n, \Diff_{\oD^n} (\oS + \ol{\Delta}))$ is lc at every $Q_i$.

\end{proof}

\begin{cor}
\label{cor:inversion_of_adjunction}
Let both $X$ and the effective divisor $S \subseteq X$ be demi-normal schemes. Furthermore, let $\Delta \geq 0$ a $\bQ$-divisor on $X$, which avoids the codimension 0 points of $S$, and the singular codimension 1 points of $X$ and $S$. Assume also that $K_X + S + \Delta$ is $\bQ$-Cartier. Then 
\begin{equation*}
 (S,\Diff_S(\Delta)) \textrm{ is slc } \Leftrightarrow (X, S+ \Delta) \textrm{ is slc near } S .
\end{equation*}
\end{cor}

\begin{proof}
Let $\rho\colon (S^n,E) \to S$ be the normalization of $S$. Similarly to \eqref{eq:inversion_of_adjunction:two}, using \cite[(5.7.2)]{Kollar_Singularities_of_the_minimal_model_program}, one can show that
\begin{equation}
\label{eq:inversion_of_adjunction:diff}
\Diff_{S^n}(\Delta)= \rho^* \Diff_S(\Delta) + E .
\end{equation}
That is, the following diagram of implications conclude our proof. 
\begin{equation*}
\xymatrix{
 (S,\Diff_S(\Delta)) \textrm{ is slc } \ar@{<=>}[rrrr]^{\textrm{\cite[Definition-Lemma 5.10]{Kollar_Singularities_of_the_minimal_model_program}}} 
& & & & (S^n, E + \rho^* \Diff_S(\Delta)) \textrm{ is lc } \ar@{<=>}[d]^{\textrm{\eqref{eq:inversion_of_adjunction:diff}}} \\
(X,S + \Delta) \textrm{ is lc near } S \ar@{<=>}[rrrr]_{\textrm{Lemma \ref{lem:inversion_of_adjunction}}}
& & & & 
(S^n,  \Diff_{S^n}(\Delta)) 
}
\end{equation*}

\end{proof}

Finally, the next lemma shows that the total space of a family of slc schemes over slc schemes is slc. Note that if one has no boundary divisors then assumption \eqref{itm:stable_fibration_total_space:singular_codim_one_pts} is vacuous. Further,  assumption \eqref{itm:stable_fibration_total_space:relative_Q_Cartier} is automatically satisfied if $f$ fulfills Koll\'ar's condition. 

\begin{lem}
\label{lem:stable_fibration_total_space}
Let  $f\colon X \to Y$ be a flat family and $\Delta_X$ and $\Delta_Y$ effective $\bQ$-divisors on $X$ and $Y$, respectively.  Assume that \begin{enumerate}                                                                                                                                      
\item \label{itm:stable_fibration_total_space:base_slc} $(Y, \Delta_Y)$ is slc,
\item \label{itm:stable_fibration_total_space:singular_codim_one_pts} $\Delta_X$ avoids singular codimension one points of the fibers,
\item \label{itm:stable_fibration_total_space:relative_Q_Cartier} there is an integer $N>0$, such that $N \Delta_X$ is an integral divisor and $\omega_{X/Y}^{[N]}(N\Delta_X)$ is a line bundle (where                                                                                                                                $\omega_{X/Y}^{[N]}(N\Delta_X)= \iota_* \omega_{U/Y}^N (N \Delta_X |_U)$ for the locus $U$ where $f$ is relative Gorenstein and $N \Delta_X$ is Cartier) and 
\item \label{itm:stable_fibration_total_space:fibers_slc} $(X_y,\Delta_X|_{X_y})$ is slc for every $y \in Y$.
\end{enumerate}
Then $(X, \Delta)$ is also slc, where $\Delta:=\Delta_X + f^*\Delta_Y$.
\end{lem}

\begin{proof}
\emph{Step 1: $X$ is demi-normal.}
$X$ is $S_2$ by \cite[Lemma 4.2]{Patakfalvi_Schwede_Depth_of_F_singularities}. Furthermore, every codimension one
point $x \in X$ is either 
\begin{itemize}
\item a smooth point of a fiber over a smooth point 
\item a nodal point of a fiber over a smooth point or
\item  a smooth point of a fiber over a nodal point. 
\end{itemize}
In either case $x$ is a nodal point. 

\emph{Step 2: $K_X + \Delta$ is $\bQ$-Cartier.}
By possibly increasing $N$ we may assume that $N(K_Y + \Delta_Y)$ is Cartier. Consider then the line bundle
\begin{equation}
\label{eq:stable_fibration_total_space:log_canonical_divisor}
f^* \left( \omega_Y^{[N]} (N \Delta_Y) \right) \otimes \omega_{X/Y}^{[N]}(N \Delta_X).
\end{equation}
By throwing out codimension at most two closed subsets we may find an open set $V \subseteq X$ such that $f|_V$ and $Y|_{f(V)}$ are Gorenstein and $N \Delta_X|_V$ and $N \Delta_Y|_{f(V)}$  are Cartier. Then we see that the line bundle \eqref{eq:stable_fibration_total_space:log_canonical_divisor} is isomorphic over $V$ to $\sO_X(N(K_X + \Delta))$. However,  since both $\sO_X(N(K_X + \Delta))$ and the line bundle \eqref{eq:stable_fibration_total_space:log_canonical_divisor} are $S_2$ sheaves, they are isomorphic by \cite[Theorem 1.12]{Hartshorne_Generalized_divisors_on_Gorenstein_schemes}. This shows that $K_X + \Delta$ is $\bQ$-Cartier indeed.


\emph{Step 3: the discrepancies are at least $-1$.}
We prove this by induction on $d : = \dim Y$. For $d=0$, $X$
coincides with its only fiber, hence all the statements are immediate. So, it is enough to show the inductional step. 

\emph{Step 3.a: the inductional step, when $(Y,\Supp \Delta_Y)$ is log-smooth.}
First, we show the inductional step when $Y$ is smooth and $\supp \Delta_Y$ has simple normal crossings. It is enough to prove that $(X, f^* \Delta_Y + \Delta_X)$ is slc near every point $x \in X$. So fix $x \in X$ and let $y := f(x)$.   Let $\Delta_Y = \sum_{i=1}^r a_i \Delta_i$, where $\Delta_i$ are distinct prime divisors, and $a_i \neq 0$. Since increasing $a_i$ does not decrease the discrepancies and $\Delta_i$ are Cartier divisors, we may assume that $a_i=1$ for every $i$. Furthermore, since we work locally around $x$ we may also assume that $y \in \Delta_i$ for all $i$. Then since adding more divisors does not decrease the discrepancies, by possibly further restricting around $x$, we may also assume that $r=d$. That is, there are $d$ components of $\Delta_Y$ meeting in normal crossings at $y$. Define then $\Gamma:= f^* (\Delta_Y - \Delta_1)$, and $S:= f^*\Delta_1$. By the inductional hypothesis, $(S, \Gamma + \Delta_X|_S)$ is slc. Then we may apply Corollary 
\ref{cor:inversion_of_adjunction} to $(X, S + \Gamma + \Delta_X)$ to obtain that so is $(X, \Delta)$. This finishes the proof of step 3.a.

\emph{Step 3.b: when $(Y,\Delta_Y)$ is log canonical.} Take a crepant log-resolution $\nu \colon (\tilde{Y},\Delta_{\tilde{Y}}) \to (Y,\Delta_Y)$ (i.e., which satisfies $\nu^*(K_Y + \Delta_Y ) = K_{\tilde{Y}} + \Delta_Y$ \cite[Notation 2.6]{Kollar_Singularities_of_the_minimal_model_program}). Note that then $(\tilde{Y},\Delta_{\tilde{Y}})$ is log canonical and $(\tilde{Y}, \supp \Delta_{\tilde{Y}})$ is log-smooth. Let $\tilde{X}:= X \times_Y \tilde{Y}$, $\tilde{f}:=f \times_Y \tilde{Y}$ and $\tilde{\nu}:= \nu \times_Y X$. First we claim that the assumptions of the lemma hold also for $(\tilde{Y}, \Delta_{\tilde{Y}})$ and $(\tilde{X}, \Delta_{\tilde{X}})$, where $\Delta_{\tilde{X}}:= \tilde{\nu}^* \Delta_X$. Indeed, the only thing that has to be checked is that $\omega_{\tilde{X}/\tilde{Y}}^{[N]}(N \Delta_{\tilde{X}})$ is a line bundle. However, this sheaf agrees in relative codimension one with $\tilde{\nu}^* \omega_{X/Y}^{[N]}(N \Delta_X)$, which is a line bundle. Further, it is reflexive by 
\cite[Corollary 3.7]{Hassett_Kovacs_Reflexive_pull_backs} and then  isomorphic to $\tilde{\nu}^* \omega_{X/Y}^{[N]}(N \Delta_X)$ by \cite[Proposition 3.6.2]{Hassett_Kovacs_Reflexive_pull_backs}. This proves our claim and then by the previous point $(\tilde{X}, \tilde{\Delta} ) :=(\tilde{X},
\tilde{f}^* \Delta_{\tilde{Y}} + \Delta_{\tilde{X}})$ is slc. Consider then the following 
stream of equalities, where we assume a compatible choice of canonical and relative canonical divisors.
\begin{align*}
 K_{\tilde{X}} + \tilde{\Delta} 
& = K_{\tilde{X}/\tilde{Y}} + \Delta_{\tilde{X}} +  \tilde{f}^*(K_{\tilde{Y}} + \Delta_{\tilde{Y}}) \\
& = \tilde{\nu}^* (K_{X/Y} + \Delta_X) + \tilde{f}^*\nu^* (K_{Y} + \Delta_Y) 
\\ & = \tilde{\nu}^* ( K_{X/Y} + \Delta_X + f^*( K_{Y} + \Delta_Y) ) 
\\ & = \tilde{\nu}^* ( K_{X} +  \Delta )
\end{align*}
This shows that $(X, \Delta)$ is slc as well, using \cite[Lemma 2.30]{Kollar_Mori_Birational_geometry_of_algebraic_varieties} and \cite[Definition-Lemma 5.10]{Kollar_Singularities_of_the_minimal_model_program}.

\emph{Step 3.c: when $(Y,\Delta_Y)$ is slc.} Let $\pi \colon (\oY,D) \to Y$ be the normalization of $Y$. Define $\oX:=X \times_Y \oY$, $E := X \times_Y \oD$, $\overline{\Delta}_Y:=\pi^* \Delta_Y$, $\of:=f \times_Y \oY$, $\tilde{\pi}:= \pi \times_Y X$ and $\overline{\Delta}_X:=\tilde{\pi}^* \Delta_X$. Similarly to as in the previous point, the assumptions of the lemma hold for $(\oX, \overline{\Delta}_X)$ and $(\oY, D + \overline{\Delta}_Y)$.  Further by the statement of the previous point, $(\oX, \of^* (D + \overline{\Delta}_Y) + \overline{\Delta}_X)=(X, E + \tilde{\pi}^* \Delta)$ is slc. Let $\rho\colon (\tilde{X},F) \to \oX$ be then the normalization of $\oX$. Note that $\tilde{\pi} \circ \rho$ is also a normalization of $X$ with  conductor divisor $F + \rho^* E$. Then the following holds using \cite[Definition-Lemma 5.10]{Kollar_Singularities_of_the_minimal_model_program} twice.
\begin{align*}
(X,  \Delta) \textrm{ is slc } 
& \Leftrightarrow \\
(\tilde{X}, F + \rho^* E + (\tilde{\pi} \circ \rho)^*  \Delta) \textrm{ is lc}
& \Leftrightarrow \\
(X, E + \tilde{\pi}^*  \Delta) \textrm{ is slc}
\end{align*}
However, we know that $(X, E + \tilde{\pi}^* \Delta)$ is slc by Step 3.b, as we have mentioned already. This finishes our proof. 
\end{proof}

\subsection{Definition of $F$}

This section contains the definition of the forgetful morphism $F$ of Theorem \ref{thm:main}, using Lemma \ref{lem:stable_fibration_total_space} from Section \ref{sec:adjunction}.  The statements of Section \ref{sec:adjunction} tell us that the composition $f$ of a fibration of stable varieties as in \eqref{eq:fibration_of_stable_schemes} have stable fibers. Here we check that Koll\'ar's condition (c.f., Definition \ref{defn:stable_family}) also holds for $f$. We  start with auxiliary statements.

\begin{lem}
\label{lem:relatively_S_d}
Let $f \colon \sX \to \sY$ and $g \colon \sY \to \sZ$ be flat morphisms of noetherian DM-stacks, and $\sF$ and $\sG$ coherent sheaves on $\sX$ and $\sY$, respectively. Further assume that $\sX$ and $\sY$ are flat and relatively $S_d$ over $\sY$ and $\sZ$, respectively. Then $\sF \otimes f^* \sG$ is flat and relatively $S_d$ over $\sZ$.
\end{lem}

\begin{proof}
First note that by passing to \'etale atlases we may assume that all stacks are schemes. Second, we show that $\sF \otimes f^* \sG$ is flat over $\sZ$. Consider an embedding $\sI \to \sO_{\sZ}$. Then by flatness of $\sG$ over $\sZ$, $\sG \otimes g^* \sI \to \sG \otimes g^* \sO_Z \cong \sG$ is an injection. However, then by flatness of $\sF$ over $\sY$ the following map is injective as well, which concludes flatness by \cite[Proposition 9.1A.a]{Hartshorne_Algebraic_geometry}.
\begin{equation*}
(\sF \otimes f^* \sG) \otimes f^* g^* \sI \cong \sF \otimes f^* (\sG \otimes g^* \sI) \to \sF \otimes f^* \sG \cong (\sF \otimes f^* \sG)  \otimes f^* g^* \sO_Z 
\end{equation*}
Finally apply \cite[Lemma 4.2]{Patakfalvi_Schwede_Depth_of_F_singularities} to obtain the statement about the relative $S_d$ property. 

\end{proof}

\begin{lem}
\label{lem:tower_of_stable_schemes_relative_canonical_relative_S_2}
Given a fibration of stable varieties as in \eqref{eq:fibration_of_stable_schemes}, $\sO_{X_i}$ and $\omega_{X_i/X_{i-1}}^{[m]}$ are flat and relatively $S_2$ over $X_j$ for every $0 \leq j < i \leq 2$, $m \in \bZ$.
\end{lem}

\begin{proof}
The statement is immediate for $\sO_{X_i}$ using Lemma \ref{lem:relatively_S_d}. For $\omega_{X_i/X_{i-1}}^{[m]}$ first we show the statement for $j=i-1$. Since $f_i$ is a family of stable varieties, flatness follows from Definition \ref{defn:stable_family}. It also follows from Definition \ref{defn:stable_family}, that $\omega_{X_i/X_{i-1}}^{[m]}|_{F} \cong \omega_F^{[m]}$ for every fiber $F$ of $f_i$. However, since $F$ is $S_2$ and $G_1$, the reflexive hull $\omega_F^{[m]}$ is $S_2$ as well \cite[Theorem 1.9]{Hartshorne_Generalized_divisors_on_Gorenstein_schemes}. This concludes the statement for $j=i-1$. For $j<i-1$, use Lemma \ref{lem:relatively_S_d}.
\end{proof}

\begin{lem}
\label{lem:relative_dualizing_reflexive}
Given a fibration of stable varieties as in \eqref{eq:fibration_of_stable_schemes},
\begin{equation*}
\omega_{X/Y}^{[m]} \otimes g^* \omega_{Y/Z}^{[m]} \cong \omega_{X/B}^{[m]}. 
\end{equation*}
Furthermore, it is flat and relatively $S_2$ over $B$.
\end{lem}

\begin{proof}
By Lemma \ref{lem:tower_of_stable_schemes_relative_canonical_relative_S_2} and  Lemma \ref{lem:relatively_S_d},  $\omega_{X/Y}^{[m]} \otimes g^* \omega_{Y/B}^{[m]}$ is flat and relatively $S_2$ over $B$. Furthermore it is isomorphic to $\omega_{X/B}^{[m]}$ in relative codimension one. Hence, \cite[Proposition 3.6]{Hassett_Kovacs_Reflexive_pull_backs}, concludes our proof.
\end{proof}

\begin{lem}
\label{lem:nef}
Given a family $X \to B$ of stable varieties, $\omega_{X/B}$ is nef (as a $\bQ$-line bundle).
\end{lem}

\begin{proof}
By \cite[Theorem 1.8]{Fujino_Semi_positivity_theorems_for_moduli_problems}, $f_*  \omega_{X/B}^{[m]}$ is a nef vector bundle for divisible enough $m$. Since $\omega_{X/B}$ is relatively ample,  $\omega_{X/B}$ is a relatively globally generated line bundle for  divisible enough $m$. Choose then an $m$, for which both hold. Then there is a surjection $f^* f_* \omega_{X/B}^{[m]} \to \omega_{X/B}^{[m]}$ from a nef vector bundle. Therefore, $\omega_{X/B}^{[m]}$ and hence $\omega_{X/B}^{[m]}$ is nef. 
\end{proof}

\begin{lem}
\label{lem:forgetful_yields_stable_family}
Given a fibration of stable varieties as in \eqref{eq:fibration_of_stable_schemes}, $f$ is a family of stable varieties. 
\end{lem}

\begin{proof}
By Lemma \ref{lem:stable_fibration_total_space}, the fibers of $f$ are slc schemes. Clearly they are proper, connected and equidimensional as well. Next we prove  that $\omega_{X/B}$ is a relatively ample $\bQ$-line bundle. Indeed, by Lemma \ref{lem:relative_dualizing_reflexive}, $\omega_{X/B} \cong g^* \omega_{Y/B} \otimes \omega_{X/Y}$. Further $\omega_{Y/B}$ is relatively ample the definition of a family of stable varieties and  $\omega_{X/Y}$ is nef and relatively ample over $Y$. Then it follows that $\omega_{X/B}$ is relatively ample as well, which implies that the fibers of $f$ are stable varieties. 

Finally we have to prove that $\omega_{X/B}^{[m]}$ is flat and compatible with arbitrary base-change. By \cite[Proposition 3.6 and Corollary 3.8]{Hassett_Kovacs_Reflexive_pull_backs} this follows if $\omega_{X/B}^{[m]}$ is flat and relatively $S_2$, which we know from Lemma \ref{lem:relative_dualizing_reflexive}.
\end{proof}

\begin{defn}
Let $\um=(m_1,m_2)$ be a dimension vector and set $m:= m_1 + m_2$. Define then $F \colon \fF \fM_{\um} \to \ofM_m$ to be the functor that takes a fibration of stable varieties as in \eqref{eq:fibration_of_stable_schemes} to the family of stable varieties $f \colon X \to B$. This latter family is indeed a family of stable varieties by Lemma \ref{lem:forgetful_yields_stable_family}. The action of $F$ on the arrows is the natural one. 
\end{defn}

\section{Deformation theory of stable fibrations}
\label{sec:deformation}

\subsection{Basic definitions}

The main technical difficulty about the deformation theory of $\ofM_h$ is that by Definition \ref{defn:stable_family} not all families with stable fibers are allowed in the pseudo-functor of $\ofM_h$. The allowed deformations are sometimes called $\bQ$-Gorenstein deformations in the literature. Another, equivalent approach is to define the index-one covering stack $\sX$ of a stable variety  $X$ and identify the deformation theory of $X$ in $\ofM_h$ by the (unconstrained) deformation theory of $\sX$ \cite{Abramovich_Hassett_Stable_varieties_with_a_twist}. We implement an analolgue of the latter approach for fibrations of stable varieties. Doing that we are forced to use the theory and language of stacks \cite{Laumon_Moret_Bailly_Champs_algebrique,stacks-project}

First let us recall the necessary definitions and facts from \cite{Abramovich_Hassett_Stable_varieties_with_a_twist}. We state the definitions of \cite{Abramovich_Hassett_Stable_varieties_with_a_twist} only in the special case when polarization is given by the canonical sheaf, and we also adapt them slightly to this situation.

\begin{defn}
\label{defn:stable_stack}
A DM-stack $\sX$ is \emph{cyclotomic}, if all its stabilizers are isomorphic to cyclotomic groups. A line bundle $\sL$ on a DM-stack $\sX$ is called \emph{uniformizing}, if $\Spec_{\sX} \left( \bigoplus_{m \in \bZ} \sL^m \right)$ is representable (by an algebraic space). If $\sX \to \sB$ is a morphism of DM-stacks, then $\sL$ is called \emph{uniformizing over $\sB$} or \emph{relatively uniformizing}, if the morphism $\Spec_{\sX} \left( \bigoplus_{m \in \bZ} \sL^m \right) \to \sB$ is representable (by algebraic spaces).  A \emph{stable stack} is a cyclotomic DM-stack $\sX$, such that
\begin{itemize}
\item $\sX$ is connected and has slc singularities (in particular it is of finite type over $k$, $S_2$, reduced, nodal in codimension one and equidimensional),
\item $\sX$ is separated,
\item $\omega_{\sX}$ is a uniformizing, ample line bundle on $\sX$ and
\item the coarse moduli map $\pi \colon \sX \to X$ is isomorphism in codimension one.
\end{itemize}
A \emph{family of stable stacks} is a flat morphism $\sX \to \sB$ of DM-stacks, such that, all $\sX_b$ are stable stacks (where $b$ is a $k$-point of $\sB$), and $\omega_{\sX/\sB}$ is a uniformizing line bundle for $\sX$ over $\sB$.  
\end{defn}

\begin{defn}
\label{defn:index_one_stack}
If $X \to B$ is a family of stable varieties, then the \emph{index-one covering stack} is defined as 
\begin{equation*}
\sX := \left[ \Spec_X \left( \bigoplus_{m \in \bZ} \omega_{X/B}^{[m]} \right) / \bG_m \right] .
\end{equation*}
\end{defn}

\begin{thm} \cite[Theorem 5.3.6]{Abramovich_Hassett_Stable_varieties_with_a_twist}
\label{thm:abramovich_hassett}
The  category $\ofM_n$ of Notation \ref{notation:stable_moduli_space} is equivalent to the category $\mathfrak{Stab}_{n}$ of families of stable stacks over $k$ of dimension $n$. The isomorphism is given by the above functors
\begin{equation*}
\begin{matrix}
\ofM_n(B) & \to & \mathfrak{Stab}_{n}(B) \\
X & \mapsto & \left[ \Spec_X \left( \bigoplus_{m \in \bZ} \omega_{X/B}^{[m]} \right) / \bG_m \right] ,
\end{matrix}
\end{equation*}
and
\begin{equation*}
\begin{matrix}
\mathfrak{Stab}_{n}(B) & \to & \ofM_n(B) \\
\sX & \mapsto & \textrm{the coarse moduli space $X$ of $\sX$} .
\end{matrix}
\end{equation*}
\end{thm}

\begin{defn}
If $X$ is a stable variety $X$, then the deformation functor of $X$ in $\ofM_h$ is denoted by $\Def_{\bQ}(X)$. That is, $\Def_{\bQ}(X)$ assigns to a local Artinian ring $A$ the set of families of stable varieties over $\Spec A$ that restrict to $X$ over the closed point of $\Spec A$. This agrees with the set of flat deformations of the scheme $X$ obeying Koll\'ar's condition form Definition \ref{defn:stable_family}. By Theorem \ref{thm:abramovich_hassett} it also agrees with the set of flat deformations of the index-one cover $\sX$ of $X$, or shortly $\Def_{\bQ}(X)=\Def(\sX)$. Notice that here we used the fact that a flat deformation of a stable stacks over a local Artinian ring is automatically a family of stable stacks. Indeed, the representability condition in Definition \ref{defn:stable_stack} is decided at geometric points by \cite[Lemma 4.4.3]{Abramovich_Vistoli_Compactifying_the_space_of_stable_maps}.
\end{defn}

The goal of Section \ref{sec:deformation} is to prove an analogue of Theorem \ref{thm:abramovich_hassett} for fibrations of stable varieties. The main issue will be to find a stacky object that encodes all $\bQ$-Gorenstein deformations of a fibration of stable varieties. Unfortunately, it will be somewhat lengthy to prove that this is indeed the case. Further technical difficulties will arise from the fact that the obtained stack cover of $X$ will be slightly different from the index-one-cover.
The above mentioned stacky object is  as follows.

\begin{defn}
\label{defn:tower_of_stable_stacks}
A \emph{fibration of stable stacks} $\usX=(\sX_i, \tf_i)$ is a commutative diagram
\begin{equation}
\label{eq:tower_of_stable_stacks}
\begin{split}
\xymatrix{
     \sX = \sX_2 \ar@/^2pc/[rr]^{\tilde{f}}  \ar[r]^{\tg=\tf_{2}} & \sX_{1} \ar[r]^{\tf_{1}} 
 & \sX_0 = B
} ,
\end{split}
\end{equation} 
where all $\tf_i$ are families of stable stacks. The \emph{coarse fibration} of a fibration of stable stacks as in \eqref{eq:tower_of_stable_stacks} is the fibration formed by the coarse  moduli spaces $X_i$ of $\sX_i$, shown in the following commutative diagram.
\begin{equation}
\label{eq:tower_of_stable_stacks_coarse_maps}
\begin{split}
\xymatrix{
    \sX =  \sX_2 \ar@/^2pc/[rr]^{\tilde{f}}  \ar[r]^{\tg=\tf_{2}} \ar[d]^{\gamma} & \sY= \sX_{1} \ar[r]^{\tf_{1}} \ar[d]^{\pi} &
 \sX_0 = B \ar@{=}[d] \\
     X= X_2 \ar@/_2pc/[rr]^{f}  \ar[r]^{g=f_{2}} & Y= X_{1} \ar[r]^{f_1} &
 X_0 = B
} ,
\end{split}
\end{equation} 
A fibration of stable stacks as in \eqref{eq:tower_of_stable_stacks}, is \emph{admissible},  if for all sufficiently  divisible $m$, the sheaves $\pi_* \tg_* \omega_{\sX/\sY}^m$  are locally free, where $\pi$ is the morphism of \eqref{eq:tower_of_stable_stacks_coarse_maps}.  

\end{defn}
So, the main goal of the section is to prove an analogue of Theorem \ref{thm:abramovich_hassett} for fibrations.  Similarly to Theorem \ref{thm:abramovich_hassett},  we obtain a fibration of schemes from a fibration of stable stacks by taking coarse moduli spaces as in \eqref{eq:tower_of_stable_stacks_coarse_maps}. To guarantee that this fibration of schemes is a fibration of stable varieties, we need the admissibility condition of Definition \eqref{defn:tower_of_stable_stacks}. Loosely speaking it guarantees that  $\Proj_{\sY} \left( \bigoplus_{m \geq 0 } \tg_* \omega_{\sX/\sY}^m \right)$, which is in certain sense a
relative coarse moduli space of $\sX$ over $\sY$, is the pullback of $X$ via $\sY \to Y$. See the proof of Lemma \ref{lem:coarse_tower_is_a_tower_of_stable_schemes} and Remark \ref{rem:coarse_tower_with_proj} for details.

Similarly when passing from a fibration of stable varieties $(X \to Y \to B)$ to a fibration of stable stacks $(\sX \to \sY \to B)$, we cannot simply take $\sX$ to be the  index-one covering stack of $X$, since then $\tg$ would not be a family of stable stacks. What we can do is the following.

\begin{defn}
\label{defn:index_one_tower}
Given a fibration of stable varieties  as in \eqref{eq:fibration_of_stable_schemes}, define its \emph{index-one cover} as the fibration of stable stacks $(\sX \to \sY \to B)$, where $\sY$ is the index-one cover of $Y$ over $B$ and
\begin{equation*}
\sX:= \left[  \Spec_{X} \left( \bigoplus_{m \in \bZ} \omega_{X/Y}^{[m]} \right) / \bG_m \right] \times_{Y} \sY . 
\end{equation*}
This definition does make sense according to Lemma \ref{lem:index_one_tower_is_a_tower_of_stable_stacks}. We usually denote the natural morphisms $\sX \to X$ and $\sY \to Y$ by $\gamma$ and $\pi$, respectively. 
\end{defn}

\subsection{Auxiliary statements}

To prove the fibration version of Theorem \ref{thm:abramovich_hassett} we need a few shorter technical statements.

\begin{lem}
\label{lem:stack_coarse_pushforward}
Let $\sX$ be a separated Deligne-Mumford stack over the scheme $U$ and $\sF$ a flat coherent sheaf on $\sX$. Denote by $\pi \colon \sX \to X$ the coarse moduli map. Then
\begin{enumerate}
\item \label{itm:stack_coarse_pushforward:flat}
$\pi_* \sF$ is flat and
\item \label{itm:stack_coarse_pushforward:S_r}
if $\sF$ is also relatively $S_r$ with relatively pure dimensional support so is $\pi_* \sF$.
\end{enumerate}
\end{lem}

\begin{proof}
We prove the two statements at once. By \cite[Lemma 2.2.3]{Abramovich_Vistoli_Compactifying_the_space_of_stable_maps}, we may assume that $\sX$ is a quotient stack $[V / G]$ for some finite group $G$, and $X$ is the scheme theoretic quotient $V/G$. Let $\rho \colon V \to [V /G]$ be the natural map. Then, by the characteristic zero assumption, 
the normalization of the  trace map $\rho_*  \sO_V \to \sO_{\sX}$ \cite[5.6, 5.7]{Kollar_Mori_Birational_geometry_of_algebraic_varieties} splits  the natural inclusion $\sO_{\sX} \to \rho_*  \sO_V$ (recall that $\rho$ if flat because it is \'etale, so $\rho_* \sO_V$ is locally free and then the trace map does make sense).
Since $\rho \colon V \to [V/G]$ is \'etale, $\rho^* \sF$ is flat (resp. flat and relatively $S_r$) over $U$. Further, since $\pi \circ \rho$ is finite, it is also affine and therefore $\pi_* \rho_* \rho^* \sF$ is flat over $U$ (resp. by the base-change property of pushforward via a finite morphism and by \cite[Proposition 5.4]{Kollar_Mori_Birational_geometry_of_algebraic_varieties}, $ \pi_* \rho_* \rho^* \sF$ is flat and relatively $S_r$ over $U$) as well. Furthermore by the above mentioned trace splitting, $\rho_* \rho^* \sF$ contains $\sF$ as a direct summand. Hence, $\pi_* \rho_* \rho^* \sF$ contains $\pi_* \sF$ as a direct summand and then consequently the latter is flat (resp. flat and relatively $S_r$)  as well. 
\end{proof}

\begin{lem}
\label{lem:tower_of_stable_stacks_relative_canonical_relative_S_2}
Given a fibration of stable stacks as in \eqref{eq:tower_of_stable_stacks}, $\sO_{\sX_i}$ and $\omega_{\sX_i/\sX_{i-1}}^{m}$ are flat and relatively $S_2$ over $\sX_j$ for every $0 \leq j < i \leq 2$, $m \in \bZ$.
\end{lem}

\begin{proof}
The statement is immediate for $\sO_{\sX_i}$ using Lemma \ref{lem:relatively_S_d}, and then also for the other sheaves, since $\omega_{\sX_i/\sX_{i-1}}^{m}$ are locally free. 
\end{proof}

\begin{lem}
\label{lem:stacky_Hassett_Kovacs}
Let $f \colon \sX \to \sY$ be a flat morphism of DM-stacks of finite type over $k$, and $\sF$ and $\sG$ two coherent sheaves on $\sX$ both flat and relatively $S_2$ over $\sY$. Further assume that there is an open substack $\iota \colon U \to \sX$, such that $\sF|_U \cong \sG|_U$ and the relative codimension of $\sX \setminus U$ is at least two. Then $\sF \cong \sG$. 
\end{lem}

\begin{proof}
It is enough to show that the natural homomorphisms  $\sF \to \iota_* (\sF|_U) $ and $\sG \to \iota_* (\sG|_U)$ are isomorphisms. Further, since the role of $\sF$ and $\sG$ are symmetric, it is enough to show only the first one. For this, by the long exact sequence of local cohomology it is enough to show that for every \'etale chart $V$, $H^i_{Z} (V,\sF)=0$ for $i=0,1$, where $Z := (V \setminus U \times_{\sX} V)_{\red}$. In fact, by the sheaf property it is enough to exhibit one \'etale cover of every etale chart for which the above vanishing holds. In particular then we may assume that there is a commutative diagram as above, where $W$ is an \'etale chart of $\sY$.
\begin{equation*}
\xymatrix{
\sX \ar[d]^f & \ar[l]^{\rho} V \ar[d]^g \\
\sY & \ar[l]^{\pi} W
}
\end{equation*}
However, then $g$ is flat and  $\rho^* \sF$ is flat and relatively $S_2$ over $W$. Hence, using that $\codim_V Z \geq 2$, by \cite[Proposition 3.3]{Hassett_Kovacs_Reflexive_pull_backs}, we obtain that the above vanishing holds indeed.
\end{proof}

\begin{lem}
\label{lem:admissible_open}
Let $(\sX \to \sY \to B)$ be a fibration of stable stacks as in \eqref{eq:tower_of_stable_stacks} over the spectrum of a local Artinian $k$-algebra $A$. Let $P$ be the closed point of $B=\Spec A$.  If the restriction $(\sX_P \to \sY_P \to P)$ over $P$ is admissible, then $(\sX \to \sY \to B)$  is admissible as well.
\end{lem}

\begin{proof}
We use the notations of \eqref{eq:tower_of_stable_stacks_coarse_maps} during the proof. First, we claim that for $m \gg 0$, the formulation of $\pi_* \tg_* \omega_{\sX/\sY}^m$ is compatible with base change, that is, for every $B' \to B$, 
\begin{equation}
\label{eq:admissible_open:compatibility}
\left( \pi_* \tg_* \omega_{\sX/\sY}^m \right)_{B'} \cong (\pi_{B'})_* (\tg_{B'})_* \left(\omega_{\sX_{B'}/\sY_{B'}} \right)^m .
\end{equation}
Since $\omega_{\sX/\sY}$ is a $\tg$-ample line bundle, for all $m \gg 0$, its  higher cohomologies on the fibers of $\tg$ vanish. In particular, then by cohomology and base change \cite[Theorem A]{Hall_Cohomology_and_base_change_for_algebraic_stacks}
\begin{equation*}
\left( \tg_* \omega_{\sX/\sY}^m \right)_{B'} \cong (\tg_{B'})_* \left(\omega_{\sX_{B'}/\sY_{B'}} \right)^m ,
\end{equation*}
and $\tg_* \omega_{\sX/\sY}^m$ is locally free. 
Furthermore by \cite[Lemma 2.2.3]{Abramovich_Vistoli_Compactifying_the_space_of_stable_maps}, $\pi_*$ commutes with base change for any sheaf. This concludes the proof of \eqref{eq:admissible_open:compatibility}. Fix for the remainder of the proof an $m$ for which \eqref{eq:admissible_open:compatibility} holds and is divisible enough.

Notice now that by \cite[Lemma 2.3.6]{Abramovich_Hassett_Stable_varieties_with_a_twist}, $\pi_{P}  \colon \sY_{P} \to Y_{P}$ is the coarse moduli map of $\sY_{P}$. Therefore, by \eqref{eq:admissible_open:compatibility} and by the assumption that $(\sX_P \to \sY_P \to P)$ is admissible, $\pi_* \tg_* \omega_{\sX/\sY}^m \big|_{Y_P}$ is a locally free sheaf. Furthermore, since $\tg_* \omega_{\sX/\sY}^m$ is locally free, it is flat over $B$. Hence by Lemma \ref{lem:stack_coarse_pushforward}.\ref{itm:stack_coarse_pushforward:flat}, $\pi_* \tg_* \omega_{\sX/\sY}^m$ is flat over $B$. Therefore, $ \pi_* \tg_* \omega_{\sX/\sY}^m $ is a flat deformation of a locally free sheaf, which is locally free by \cite[Exercise 7.1]{Hartshorne_Deformation_theory}. This finishes our proof.
\end{proof}

\begin{lem}
\label{lem:proj_is_coarse_moduli_space}
Let $f \colon \sX \to \sY$ be a proper morphism of separated DM-stacks and $\sL$ an $f$-ample line bundle. Define $\sZ : = \Proj_{\sY} \left( \bigoplus_{n \geq 0} f_* (\sL^n) \right)$ and let $\rho\colon \sX \to \sZ $ be the natural morphism. Then $\rho_* \sO_{\sX} \cong \sO_{\sZ}$. Furthermore, if $f$ was flat, so is $\sZ$ over $\sY$.
\end{lem}

\begin{proof}
Since the question is \'etale local on $\sY$, we may assume that $Y:=\sY$ is a scheme. Let then $\pi \colon \sX \to Z$ be the  coarse moduli map of $\sX$ and $g \colon Z \to Y$ the natural induced morphism. It is enough to show that $\sZ \cong Z$, compatibly with $\rho$ and $\pi$. 

Since $\sX$ is a DM-stack, there is an integer $m>0$, and a line bundle $\sK$ on $Z$, such that $\pi^* \sK \cong \sL^m$. Then, $\pi^* \sK^n \cong \sL^{n \cdot m}$ for every $m$ and $\sK$ is also relatively ample over $Y$. Therefore, the following computation concludes our proof.
\begin{equation*}
 Z \cong \Proj_Y \left(\bigoplus_{n \geq 0} g_* (\sK^n) \right) 
\cong 
\underbrace{\Proj_Y \left(\bigoplus_{n \geq 0} g_* \pi_* (\sL^{n \cdot m}) \right) }_{ \parbox{130pt}{ \tiny projection formula and the fact that since $\pi$ is a coarse moduli map, $\pi_* \sO_{\sX} \cong \sO_Z$}}
 \cong \Proj_Y \left(\bigoplus_{n \geq 0} f_* (\sL^{n \cdot m}) \right) 
\cong \sZ
\end{equation*}
\end{proof}

\subsection{Equivalences of deformation functors}

Here we show the promised fibration version of Theorem \ref{thm:abramovich_hassett}

\begin{lem}
\label{lem:pi_is_a_coarse_moduli map}
If $(\sX \to \sY \to B)$ is the index-one cover of a fibration $(X \to Y \to B)$ of stable varieties as in Definition \ref{defn:index_one_tower}, then the natural morphisms $\pi \colon \sY \to Y$ and $\gamma \colon \sX \to X$ are coarse moduli morphisms.
\end{lem}

\begin{proof}
For $\pi$ this follows from Theorem \ref{thm:abramovich_hassett}. Hence we restrict to $\gamma$ from now. Since $\gamma$ is proper, we have to show that it is quasi-finite and $\gamma_* \sO_{\sX} \cong \sO_{X}$.  First, let us introduce some notation in the following commutative diagram (here $\sX = X \times_Y \sY$ as defined in Definition \ref{defn:index_one_tower}).
\begin{equation*}
\xymatrix{
X \ar[d]^{g} & \sX':=\left[ \Spec_{X} \left( \bigoplus_{m \in \bZ} \omega_{X/Y}^{[m]} \right) / \bG_m \right] \ar[dl]^{\ug} \ar[l]^-{\eta} & \sX \ar[l]_-{\zeta} \ar[dl]^{\tg} \ar@/_2pc/[ll]_{\gamma} \\
Y & \ar[l]^{\pi} \sY
}
\end{equation*}
First, since $\pi$ is a coarse moduli map, it is quasi-finite. Second, by Theorem \ref{thm:abramovich_hassett}, $\eta$ is a coarse moduli map of a DM-stack, hence it is also quasi-finite. So, it follows that $\gamma$ is quasi-finite. For the other condition, notice that $\eta_* \sO_{\sX'} \cong \sO_{X}$, since $\eta$ is a coarse moduli morphism. Furthermore, by flat base-change \cite[Corollary 4.2.2]{Brochard_Finiteness_theorems_for_the_Picard_objects_of_an_algebraic_stack} ($\ug$ is flat, since it is the index-one-cover of $g$ which is flat by Theorem \ref{thm:abramovich_hassett}),
\begin{equation*}
\zeta_* \sO_{\sX} \cong \zeta_* \tg^* \so_{\sO_{\sY}} \cong \ug^* \pi_* \sO_{\sY} 
\cong 
\underbrace{\ug^* \sO_{Y}}_{ \parbox{55pt}{\tiny $\pi$ is a coarse moduli map}} 
\cong
\sO_{\sX'}.
\end{equation*}
Hence $\gamma_* \sO_{\sX} \cong \sO_{X}$ and $\gamma$ is a coarse moduli morphism indeed.
\end{proof}

\begin{lem}
\label{lem:index_one_tower_is_a_tower_of_stable_stacks}
The index-one cover $(\sX \to \sY \to B)$ of a fibration $(X \to Y \to B)$  of stable varieties defined in Definition \ref{defn:index_one_tower} is indeed a fibration of stable stacks. Furthermore, it is an admissible fibration of stable stacks. 
\end{lem}

\begin{proof}

By \cite[Theorem 5.3.6]{Abramovich_Hassett_Stable_varieties_with_a_twist} and the fact that the notion of a family of stable stacks is invariant under base-change, $(\sX \to \sY \to B)$ is a fibration of stable stacks. 
To prove admissibility, first note that by Lemma \ref{lem:pi_is_a_coarse_moduli map}, $(X \to Y \to B)$ is the coarse moduli fibration of $(\sX \to \sY \to B)$.  Second, note also that by Lemma \ref{lem:tower_of_stable_stacks_relative_canonical_relative_S_2}, $\omega_{\sX/\sY}^m$ is  flat and relatively $S_2$ over $B$. Hence, by Lemma \ref{lem:stack_coarse_pushforward}.\ref{itm:stack_coarse_pushforward:S_r}, $\gamma_* (\omega_{\sX/\sY}^m)$ is flat and relatively $S_2$ over $B$ as well. Furthermore, it is isomorphic in relative codimension 1  to $\omega_{X/Y}^{[m]}$ which is also flat and relatively $S_2$ over $B$ according to Lemma \ref{lem:tower_of_stable_schemes_relative_canonical_relative_S_2}.  So, these two sheaves are isomorphic globally by \cite[Corollary 3.8]{Hassett_Kovacs_Reflexive_pull_backs}. That is,
\begin{equation*}
\pi_* \tg_* \omega_{\sX/\sY}^m \cong g_* \gamma_*  \omega_{\sX/\sY}^m \cong g_*   \omega_{X/Y}^{[m]},
\end{equation*}
which is locally free for all divisible enough $m \gg 0$. This concludes our proof.
\end{proof}

\begin{lem}
\label{lem:coarse_tower_is_a_tower_of_stable_schemes}
The coarse fibration $( X \to Y \to B)$ as in \eqref{eq:tower_of_stable_stacks_coarse_maps} of an admissible fibration of stable stacks $(\sX \to \sY \to B)$  is a fibration of stable varieties. 
\end{lem}

\begin{proof}
First, note that by Theorem \ref{thm:abramovich_hassett} $Y \to B$ is a family of stable varieties.  We need to show that so is $g \colon X \to Y$.  First, we \emph{claim that for big and divisible enough $m$}, 
\begin{equation}
\label{eq:coarse_tower_is_a_tower_of_stable_schemes:base_change}
\pi^* \pi_* \tg_* \omega_{\sX/\sY}^m \cong \tg_* \omega_{\sX/\sY}^m .
\end{equation}
Indeed, 
the shaves $\tg_* \omega_{\sX/\sY}^m$ are locally free for all $m \gg 0$. Choose an $m$ for which this holds and also $\pi_* \tg_* \omega_{\sX/\sY}^m $ is  locally free. Then, $\pi^* \pi_* \tg_* \omega_{\sX/\sY}^m $ is locally free as well, and in particular it is flat and relatively $S_2$ over $B$. Furthermore, since $\pi$ is isomorphism in relative codimension one over $B$, this sheaf is isomorphic to $\tg_* \omega_{\sX/\sY}^m$ in relative codimension one. Therefore by Lemma \ref{lem:stacky_Hassett_Kovacs} we obtain \eqref{eq:coarse_tower_is_a_tower_of_stable_schemes:base_change}. This finishes the proof of our claim.

Define then
\begin{equation*}
X := \Proj_{Y} \left( \bigoplus_{m \geq 0} \pi_* \tg_* \omega_{\sX/\sY}^m \right) 
\textrm{ and }
\sX' := \Proj_{\sY} \left( \bigoplus_{m \geq 0} \tg_* \omega_{\sX/\sY}^m \right) .
\end{equation*}
Notice  that by \eqref{eq:coarse_tower_is_a_tower_of_stable_schemes:base_change}, $\sX' \cong X \times_{Y} \sY$. Further note that by Lemma \ref{lem:proj_is_coarse_moduli_space} applied to $\pi \circ \tg \colon \sX \to Y$, $X$ is the coarse moduli space of $\sX$. We just have to prove that $f \colon X \to Y$ is a stable family. Choose now a scheme $Z$ that maps finitely and surjectively to $\sY$ via $\xi \colon Z \to \sY$. Pulling back $\sX'$ (or equivalently $X$), $\sX$ and the natural morphism  over $Z$ we obtain a diagram
\begin{equation*}
\xymatrix{
X \times_Y Z  \cong \sX' \times_{\sY} Z \cong  \Proj_{Z} \left( \bigoplus_{m \geq 0} \tg_{Z,*} \omega_{\sX_Z/Z}^m \right) \ar[d] & \ar[l] \ar[dl]^{\tg_Z} \sX \times_{\sY} Z =: \sX_Z \\
Z
} . 
\end{equation*}
Note that $\sX' := \Proj_{\sY} \left( \bigoplus_{m \geq 0} \tg_* \omega_{\sX/\sY}^m \right) $ a priori pulls back to $\Proj_{Z} \left( \bigoplus_{m \geq 0} \xi^* \tg_* \omega_{\sX/\sY}^m \right) $. However, the isomorphism class of $\Proj$ is not affected by passing to a Veronese subalgebra, and for divisible enough $m$ we have $\xi^* \tg_* \omega_{\sX/\sY}^m \cong  \tg_{Z,*} \omega_{\sX_Z/Z}^m$ by the relative ampleness of $\omega_{\sX/\sY}$ and cohomology and base-change \cite[Theorem A]{Hall_Cohomology_and_base_change_for_algebraic_stacks}. 

By Lemma \ref{lem:proj_is_coarse_moduli_space} and Theorem \ref{thm:abramovich_hassett} $X \times_Y Z \to Z$ is a stable family. So, in particular its fibers are stable varieties. Since the fibers of $g\colon X \to Y$ are also fibers of $X \times_Y Z \to Z$ we see that the fibers of the former are stable varieties. We are left to show that $g$ is flat and that condition \eqref{eq:stable_family:Kollar_condition} holds for it. For the first one, note that by flattening decomposition \cite[Lecture 8]{Mumford_Lectures_on_curves_on_an_algebraic_surface} (also \cite[Theorem 1]{Kollar_Hulls_and_Husks}) there is a locally closed $\cup_j Y^j = Y$, such that for $T \to Y$, $X \times_{Y} T$ is flat over $T$ if and only if $T \to Y$ factorizes through $Y^j$ for some $j$. Applying this to $T=Z$, the image of $Z \to Y$ has to factor through $Y^j$ for some $j$.  Therefore, by the surjectivity of $Z \to Y$, $Y^j = Y$ for some $j$ and therefore $g$ is flat. Condition \eqref{eq:stable_family:Kollar_condition} is shown 
similarly but 
using the locally closed decomposition given by  \cite[Corollary 25]{Kollar_Hulls_and_Husks}.

\end{proof}

\begin{rem}
\label{rem:coarse_tower_with_proj}
Note that the proof of Lemma \ref{lem:coarse_tower_is_a_tower_of_stable_schemes} yields also that if $(\sX \to \sY \to B)$ is an admissible fibration of stable stacks, then the coarse fibration $( X \to Y \to B)$  can be described using the notations of
\eqref{eq:tower_of_stable_stacks_coarse_maps} as
\begin{equation*}
X \cong \Proj_{Y} \left( \bigoplus_{m \geq 0} \pi_* \tg_* \omega_{\sX/\sY}^m \right) .
\end{equation*}
Furthermore, similarly $\sX':= X \times_{Y} \sY$ can be described as
\begin{equation*}
\sX' := \Proj_{\sY} \left( \bigoplus_{m \geq 0} \tg_* \omega_{\sX/\sY}^m \right),
\end{equation*}
and if $\rho \colon \sX \to \sX'$ is the natural morphism then $\rho_* \sO_{\sX} \cong \sO_{\sX'}$. In particular, $\rho$ becomes a coarse moduli map after pulling back via any finite or \'etale cover of $\sY$ by a scheme. 
\end{rem}

The following theorem is the promised fibration version of Theorem \ref{thm:abramovich_hassett}. The previous two lemmas guarantee that the two functors in the statement do make sense.

\begin{thm}
\label{thm:tower_equivalence}
There is an equivalence of the category $\fF \fM_{\um}$ of fibrations of stable varieties of dimension vector $\um$ introduced in Definition \ref{defn:fibration_of_stable_varieties}  and of the category of admissible fibrations of stable stacks $\mathfrak{Fibr}_{\um}$ with the same dimension vector given by the above functors
\begin{equation}
\label{eq:tower_equivalence:scheme_to_stack}
\begin{matrix}
\fF \fM_{\um}(B) & \to & \mathfrak{Fibr}_{\um}(B) \\
(X \to Y \to B) & \mapsto & (\sX \to \sY \to B) = \textrm{the index-one cover of } (X \to Y \to B),
\end{matrix}
\end{equation}
and
\begin{equation}
\label{eq:tower_equivalence:stack_to_scheme}
\begin{matrix}
\mathfrak{Fibr}_{\um}(B) & \to & \fF \fM_{\um}(B) \\
(\sX \to \sY \to B) & \mapsto & (X \to Y \to B) = \textrm{the coarse fibration of } (\sX \to \sY \to B).
\end{matrix} 
\end{equation}
\end{thm}

\begin{rem}
Recall that a morphisms of stacks is a functor and two morphisms are said to be equivalent if the corresponding functors are naturally isomorphic. When building a moduli space of stacks, it can be useful to remember these natural isomorphism as well, thus obtaining a 2-category where arrows can also have isomorphisms.  However in the case of $\mathfrak{Fibr}_{\um}$ the 2-category approach turns out to be unnecessary, because no arrow 
\begin{equation}
\label{eq:arrow}
(\sX \to \sY \to B) \to (\sX' \to \sY' \to B') 
\end{equation}
between fibrations of stable stacks as in \eqref{eq:tower_of_stable_stacks} has non-trivial automorphisms (and hence isomorphisms between arrows are unique if they exist). Indeed, by \cite[Lemma 4.2.3]{Abramovich_Vistoli_Compactifying_the_space_of_stable_maps},  $\sX \to \sX'$ and $\sY \to \sY'$ do not have non-trivial automorphisms and then it follows that also \eqref{eq:arrow} does not have any. Using the categorical language,  the 2-category $\mathfrak{Fibr}_{\um}$ is equivalent to a $1$-category.
\end{rem}

\begin{proof}[Proof of Theorem \ref{thm:tower_equivalence}]

\emph{``Step 1: \eqref{eq:tower_equivalence:scheme_to_stack} applied first and then \eqref{eq:tower_equivalence:stack_to_scheme}'' is naturally isomorphic to identity.} We have to show that the coarse moduli space of $\sX$ and $\sY$, defined in \eqref{eq:tower_equivalence:scheme_to_stack}, is $X$ and $Y$. However, this has already been shown in lemma \ref{lem:pi_is_a_coarse_moduli map}.

\emph{Step 2: ``\eqref{eq:tower_equivalence:stack_to_scheme} applied first and then \eqref{eq:tower_equivalence:scheme_to_stack}'' is naturally isomorphic to identity.}  Given an admissible fibration of stable stacks $(\sX \to \sY \to B)$, if $X$ and $Y$ the coarse moduli spaces of $\sX$ and $\sY$ as in \eqref{eq:tower_equivalence:stack_to_scheme}, and $(\sX_0 \to \sY_0 \to B)$ is the index-one cover of the tower of families of stable varieties $(X \to Y \to B)$  we are supposed to prove that  $(\sX_0 \to \sY_0 \to B)$ is isomorphic to $(\sX \to \sY \to B)$. The isomorphism of $(\sY_0 \to B )$ to $(\sY \to B)$ immediately follows from Theorem \ref{thm:abramovich_hassett}. 

Hence we have to show that $(\sX \to \sY)$ is isomorphic to $(\sX_0 \to \sY_0)$. Since we have already identified $\sY$ with $\sY_0$, really we have to show that $\sX_0 \to \sY$ is isomorphic to $\sX \to \sY$, where 
\begin{equation*}
\sX_0 := \left[ \Spec_{X} \left( \bigoplus_{m \in \bZ} \omega_{X/Y}^{[m]} \right) / \bG_m \right] \times_{Y} \sY .
\end{equation*}
By Theorem \ref{thm:abramovich_hassett}, the stack quotient in the above formula is a family of stable stacks. Hence, since the notion of a family of stable stacks is pullback invariant,  both $\sX_0$ and $\sX$ are families of stable stacks over $\sY$. Further, over the relative Gorenstein locus $\sY_{\Gor}$, where $\sY \to Y$ is an isomorphism, $\sX$ and $\sX_0$ are isomorphic by Theorem \ref{thm:abramovich_hassett}. The idea is to apply now that the moduli space $\mathfrak{Stab}_n$ of stable stacks is a separated DM-stack (c.f. Theorem \ref{thm:abramovich_hassett}, \cite[Proposition 6.1.4]{Abramovich_Hassett_Stable_varieties_with_a_twist}, \cite[Theorem 2.8]{Bhatt_Ho_Patakfalvi_Schnell_Moduli_of_products_of_stable_varieties}) and deduce then that $\sX$ is isomorphic to $\sX_0$ over the entire $\sY$. The only issue is that we  know the universal property of $\mathfrak{Stab}_n$ only for a map from a scheme to $\mathfrak{Stab}_n$. So, we have to pass to \'etale charts of $\sY$ to apply the above idea. 

Choose  an \'etale cover $s \colon U \to \sY$ by a scheme. Let $V:= U \times_{\sY} U$ and $p \colon V \to U$ and $q \colon V \to U$ the two projections.  We claim that \emph{there is an isomorphism  $\xi \colon \sX \times_{\sY} U \to \sX_0 \times_{\sY} U$ such that $p^* \xi = q^* \xi$}. This then implies the required isomorphism of $\sX$ and $\sX_0$ over $\sY$ since by the stack axioms, $\sX$ and $\sX_0$ glue in the \'etale toplogy. The existence  of $\xi$ follows similarly to the  argument of the previous paragraph: fix an isomorphism $\zeta\colon \sX|_{\sY_{\Gor}} \to \sX_0|_{\sY_{\Gor}}$. Then $\left(s|_{U_{\Gor}}\right)^* \zeta$ is an isomorphism of $\sX \times_{\sY} U_{\Gor}$ and $\sX_0 \times_{\sY} U_{\Gor}$, where $U_{\Gor}$ is the Gorenstein locus of $U$.  Then using that $\mathfrak{Stab}_n$ is separated and hence $\Isom_U (\sX \times_{\sY} U, \sX_0 \times_{\sY} U)$ is finite over $U$ yields that there is a unique extension of this isomorphism over the entire $U$.  Notice now that 
\begin{equation*}
p^* \xi|_{V_{\Gor}} =  \left(s \circ p|_{V_{\Gor}} \right)^* \zeta =  \left(s \circ q|_{V_{\Gor}} \right)^* \zeta = q^* \xi|_{V_{\Gor}}.
\end{equation*}
Now using the properness of $\Isom_V (\sX \times_{\sY} V, \sX_0 \times_{\sY} V)$ over $V$ implies that $p^* \xi $ and $q^* \xi$ agree over the entire $V$. This finishes concludes our claim and hence our proof as well.

\end{proof}

\subsection{Conclusion}

Using  Theorem \ref{thm:tower_equivalence}, we express explicitly what vanishing is needed to show Theorem \ref{thm:main}. The initial idea is that starting with a fibration of stable varieties $\uX=( X \to Y \to \Spec k)$ with index-one covering fibration $\usX=(\sX \to \sY \to \Spec k)$  use the following commutative diagram of deformation functors.
\begin{equation*}
\xymatrix{
\Def ( \usX) \ar[d]_{ \parbox{40pt}{\tiny forgetting the middle level}} \ar[rrr]^{\textrm{taking coarse moduli space}} & & & \Def_{\bQ}( \uX)  \ar[d]^{ \parbox{50pt}{\tiny forgetting the middle level}} & \\
\Def ( \sX ) \ar[rrr]_{\textrm{taking coarse moduli space}} &  & &  \Def_{\bQ}( X ) 
}
\end{equation*}
By Theorem \ref{thm:tower_equivalence}, the top horizontal arrow is an equivalence. In this section (in Proposition \ref{prop:equivalence}) we will also prove that the left vertical arrow is an equivalence. Then we would like to use Theorem \ref{thm:abramovich_hassett} to say that the bottom horizontal arrow is an equivalence as well, and then so is the right vertical one, which would conclude the proof of Theorem \ref{thm:main}. However, unfortunately Theorem \ref{thm:abramovich_hassett} does not apply to the lower horizontal arrow, since $\sX$ is not the index-one cover of $X$. Hence we factor the bottom arrow as
\begin{equation*}
\xymatrix{
\Def (\sX) \ar@/^2pc/[rrrrr]^{\textrm{taking coarse moduli space}} & \Def(  \sX \to \widetilde{\sX} ) \ar[l] \ar[r] &  \Def(\widetilde{\sX}) \ar[rrr]_{\textrm{taking coarse moduli space}} &  & &  \Def_{\bQ}( X ) 
}
\end{equation*}
where $\widetilde{\sX}$ is the index-one cover of $X$ and in the following proposition we show that the introduced new arrows are equivalences.

\begin{prop}
\label{prop:two_index_ones}
Given a fibration of stable varieties $(X \to Y \to  \Spec k)$,  let $(\sX \to \sY \to \Spec k)$ be the index-one cover of it as in Theorem \ref{thm:tower_equivalence} and  $\widetilde{\sX}$  the index-one cover of $X$ as in Theorem \ref{thm:abramovich_hassett}. Then there is a morphism $\phi \colon \sX \to \widetilde{\sX}$ factoring $\gamma \colon \sX \to X$, such that the following two natural functors of deformation spaces are equivalences
\begin{equation*}
 \Def(\widetilde{\sX}) \longleftarrow \Def(  \sX \to \widetilde{\sX} ) \longrightarrow  \Def(\sX).
\end{equation*}
\end{prop}

\begin{proof}
\emph{Step 1: defining $\phi$.} 
 First, we prove that
\begin{equation}
\label{eq:two_index_ones:goal}
\sX \cong \left[ \left. \Spec_{X} \sA \right/ \bG^2_m \right]  \textrm{, where } \sA := \bigoplus_{(m_1,m_2) \in \bZ^2} \left( g^* \omega_{Y}^{[m_1]} \otimes  \omega_{X /Y}^{[m_2]} \right) .
\end{equation}
Let 
\begin{equation*}
 \sB:= \bigoplus_{m \in \bZ} \omega_{X/Y}^{[m]} \textrm{ and }  \sC := \bigoplus_{m \in \bZ} \omega_{Y}^{[m]}. 
\end{equation*}
Then, the following compuation shows \eqref{eq:two_index_ones:goal}.
\begin{eqnarray*}
\sX & = & \left[ (\Spec_{X} \sB) / \bG_m \right] \times_{Y} \sY
\\ & \cong & \left[ (\Spec_{X} \sB) / \bG_m \right] \times_{X} (X \times_{Y} \sY)
\\ & \cong & \left[ (\Spec_{X} \sB) / \bG_m \right] \times_{X} \left[ \left.  (\Spec_{X} g^* \sC) \right/ \bG_m \right]
\\ & \cong & \left[ (\Spec_{X} \sB  \times_{X} \Spec_{X} g^* \sC) / \bG_m^2 \right] 
\\ & \cong & \left[ (\Spec_{X} \sB  \otimes  g^* \sC) / \bG_m^2 \right] 
\\ & \cong & \left[ (\Spec_{X} \sA) / \bG_m^2 \right] 
\end{eqnarray*}
Furthermore by Lemma \ref{lem:relative_dualizing_reflexive}, there is a (graded) embedding
\begin{equation}
\label{eq:two_index_ones:embedding}
\bigoplus_{m \in \bZ} \omega_{X}^{[m]} \hookrightarrow \sA,
\end{equation}
which induces a morphism
\begin{equation}
\label{eq:two_index_ones:before_quotient}
\Spec_{X} \sA \to \Spec_{X} \left(  \bigoplus_{m \in \bZ} \omega_{X}^{[m]}  \right).
\end{equation}
Furthermore by the grading of \eqref{eq:two_index_ones:embedding}, \eqref{eq:two_index_ones:before_quotient} is equivariant with respect to the $2$-times multiplication map $\xi \colon \bG_m^2 \to \bG_m$. Quotienting  then out with $\bG_m^2$ and $\bG_m$ on the two sides of \eqref{eq:two_index_ones:before_quotient} yields the morphism $\phi \colon \sX \to \widetilde{\sX}$.

\emph{Step 2: $\Def(  \sX \to \widetilde{\sX} ) \to \Def(\widetilde{\sX})$ is an equivalence.} We use \cite[Proposition 3.9]{Bhatt_Ho_Patakfalvi_Schnell_Moduli_of_products_of_stable_varieties}. That is, we have to exhibit an open set $U \subseteq \sX$, such that $\phi|_U$ is an isomorphism, $\codim_{\sX} \sX \setminus U \geq 3$ and $\depth \sO_{\sX, \oy}^{\mathrm{sh}} \geq 3$ for every geometric point point $\oy \in \sX \setminus U$ (where the upper index $\mathrm{sh}$ denotes the \'etale local ring, opposite to the usual Zariski one, the notation coming from ``strict Henselization'', the algebraic operation with wich one can obtain the \'etale local ring from the Zariski local ring in case of a scheme). 

Consider now any (not necessarily closed) point $x \in X$. Set $y:=g(x)$, $c_1 := \codim_{Y} y$, $c_2:= \codim_{X_y} x$ and $c:=\codim_{X} x$. Note that $c_1 + c_2 = c$. Further, note that if $c \leq 3$, then at most one of $c_j$ can be bigger than 1, and hence $x$ (resp. $y$) is a relatively Gorenstein point over $Y$ (resp. $\Spec k$). Therefore $\omega_{X/Y}$ or $g^* \omega_Y$ is a line bundle at $x$. Let $W$ be the locus of points  $x \in X$ where $\omega_{X/Y}$ or $g^* \omega_Y$ is locally free. By the above discussion $\codim_{X} X \setminus W \geq 4$. Define then $U$ and $V$ be the inverse image of $W$ in $\sX$ and in $\osX$, respectively. In particular then $\phi^{-1}(V)=U$.

First,  $\phi|_{U}$ is an isomorphism, because  after choosing an $R \in W$ one of the following two cases holds:
\begin{enumerate}
 \item If $ \omega_{X/Y}$ is locally free at $R$, then there is an open neighborhood $T$ of $R$ such that $\omega_{X/Y}|_T \cong \sO_T$. Hence, $ \left. \bigoplus_{m \in \bZ} \omega_{X}^{[m]} \right|_T \cong  \left. \bigoplus_{m \in \bZ} g^* \omega_{Y}^{[m]}\right|_T$ and 
 $\sA|_T \cong  \left. \bigoplus_{m \in \bZ}  g^* \omega_{Y}^{[m]} \right|_T[x,x^{-1}] $. Hence over $V$,  $\Spec \sA \to \Spec \left(\bigoplus_{m \in \bZ}  \omega_{X}^{[m]}  \right)  $ is a $\bG_m$-bundle. Further the restriction of the natural $\bG_m \times \bG_m$ action on $\Spec \sA$ to the kernel of the multiplication map $\bG_m \times \bG_m \to \bG_m$ acts freely and transitively on the fibers over $V$. Therefore, the map $\phi : \sX \to \widetilde{\sX}$ obtained by quotienting out $\Spec \sA \to \Spec \left(\bigoplus_{m \in \bZ}  \omega_{X}^{[m]}  \right) $ (by $\bG_m \times \bG_m$ on the source and by $\bG_m$ on the target side) is an isomorphism over $V$.

\item If $g^* \omega_Y$ is locally free at $R$, the argument is completely the same only the roles of $\omega_{X/Y}$ and $g^* \omega_Y$ are exchanged.
\end{enumerate}

Second, we have to prove that $\depth \sO_{\sX, \ox}^{\mathrm{sh}} \geq 3$  for every geometric point $\ox \in \sX \setminus U$. So, fix any such $\ox$. Since $\ox \not\in U$, $c_1, c_2 \geq 2$ . Consequently  $\depth \sO_{(\sX)_{\oy}, \ox}^{\mathrm{sh}} \geq 2$ and $\depth \sO_{\sY, \oy}^{\mathrm{sh}} \geq 2$. However, then by  \cite[Proposition 6.3.1]{Grothendieck_Elements_de_geometrie_algebrique_IV_II}, $\sO_{\sX,\ox}^{\mathrm{sh}} \geq 4$.

\emph{Step 3: $\Def(  \sX \to \widetilde{\sX} ) \to \Def(\sX)$ is an equivalence.} We use \cite[Proposition 3.10]{Bhatt_Ho_Patakfalvi_Schnell_Moduli_of_products_of_stable_varieties}. That is, we have to show that $\widetilde{\sX}$ has no infinitesimal automorphisms,  $(\phi)_* \sO_{\sX} \cong \sO_{\widetilde{\sX}}$ and that $R^1 (\phi)_* \sO_{\sX} = 0$. The first condition is is shown in Lemma \ref{lem:infinitesmial_automorphisms}. For the other two, 
consider the following diagram (recall $\xi$ is the multiplication map $\bG_m^2 \to \bG_m$).
\begin{equation*}
\xymatrix{
Q:= \Spec_{X} \sA \ar[d]^g \ar@/_5pc/[dd]_q \ar[dr]^{\rho}
\\
R:= [Q/\Ker \xi ] \ar[d]^h \ar[r]_-{\zeta} 
&
P:= \Spec_{X} \left(  \bigoplus_{m \in \bZ} \omega_{X}^{[m]}  \right) \ar[d]^p
\\
\sX = [Q/\bG_m^2]=[R/\bG_m] \ar[r]^{\phi} 
&
\widetilde{\sX} = [P / \bG_m]
}
\end{equation*}
From the definition of $P$ it follows that $P$ is isomorphic to the scheme theoretic quotient $Q/\Ker \xi$. Therefore, $\zeta$ is a coarse  moduli map. However, then 
\begin{equation*}
\sO_{\widetilde{\sX}} \cong (p_* \sO_P)^{\bG_m} 
\cong 
\underbrace{(p_*  \zeta_* \sO_R)^{\bG_m}}_{\textrm{$\zeta$ is a coarse moduli map}}
\cong
(\phi_* h_* \sO_R)^{\bG_m} = \phi_* \sO_{\sX} 
\end{equation*}
and
\begin{equation*}
 R^1 \phi_* \sO_{\sX} \hookrightarrow R^1 \phi_* h_* \sO_R 
\cong
\underbrace{R^1(\phi \circ h)_* \sO_R}_{\textrm{$h$ is affine}}
\cong
R^1(p \circ \zeta)_* \sO_R
\cong
\underbrace{R^1 p_*   \sO_P}_{ \parbox{90pt}{\tiny $\zeta$ is a coarse moduli map and hence $R \zeta_* \sO_R \cong \sO_P$}}
= \underbrace{0 }_{\textrm{$p$ is affine}}
\end{equation*}
This concludes our proof.
\end{proof}


\begin{lem}
\label{lem:infinitesmial_automorphisms}
Given a fibration of stable varieties $(X \to Y \to \Spec k)$,  let $(\sX \to \sY \to \Spec k)$ be the index-one cover  as in Theorem \ref{thm:tower_equivalence} and  $\widetilde{\sX}$  the index-one cover of $X$ as in Theorem \ref{thm:abramovich_hassett}. Then neither $\sX$, nor $\widetilde{\sX}$ has  infinitesimal automorphisms.
\end{lem}

\begin{proof}
First,  note that if $\phi \colon \sX \to \widetilde{\sX}$ is the morphism constructed in  Proposition \ref{prop:two_index_ones}, then $\phi$ factors the coarse moduli map $\gamma\colon \sX \to X$ and by the proof of Proposition \ref{prop:two_index_ones}, $\phi_* \sO_{\sX} \cong \sO_{\widetilde{\sX}}$. Therefore, it follows that the induced morphism $\widetilde{\sX} \to X$ is also a coarse moduli map. Furthermore, since $\phi$ is isomorphism over the locus $U \subseteq X$ where either $\omega_{X/Y}$ or $g^* \omega_Y$ is a line bundle, so is the morphism $\widetilde{\sX} \to X$. Hence, it is enough to prove that a DM-stack $\sZ$ with a proper coarse moduli map $\alpha \colon \sZ \to X$ which is an isomorphism over $U$ has no infinitesimal automorphisms. This will imply the statement for both $\sX$ and $\widetilde{\sX}$.

By Theorem \ref{thm:vanishing_hom_from_cotangent}, $X$ has no infinitesimal automorphism. To deduce, the same for $\sZ$, note that the map $\alpha \colon \sZ \to X$ is an isomorphism in codimension one. Hence, $\bL_{\sZ/X}$ is supported in a closed set of codimension at least two. Consider now the exact triangle 
\begin{equation*}
\xymatrix{
\tau_{\leq -1} \bL_{\sZ/X} \ar[r] & \bL_{\sZ/X} \ar[r] & \Omega_{\sZ/X} \ar[r]^{+1} & 
},
\end{equation*}
where $\tau_{\leq -1} \bL_{\sZ/X}$ is supported only in cohomological degrees smaller than zero. In particular then $\Hom_{\sZ} ( \tau_{\leq -1} \bL_{\sZ/X}  , \sO_{\sZ})=0$. Furthermore, $\Hom_{\sZ} (  \Omega_{\sZ/X} , \sO_{\sZ})=0$ because $\Omega_{\sZ/X}$ is a sheaf supported on a closed set of codimension at least two. Hence by $\Hom_{\sZ} (  \_ , \sO_{\sZ})$ applied to the above exact triangle we obtain that, $\Hom_{\sZ} (  \bL_{\sZ/X} , \sO_{\sZ})=0$. Applying now $\Hom_{\sZ} ( \_, \sO_{\sZ})$ to  the usual exact sequence of cotangent complexes associated to $\alpha$ yields the exact sequence
\begin{equation*}
\xymatrix{
\Hom_{\sZ} (  \bL_{\sZ/X} , \sO_{\sZ}) \ar[r] &  \Hom_{\sZ} ( \bL_{\sZ} , \sO_{\sZ}) \ar[r] & \Hom_{\sZ} ( L \alpha^* \bL_{X} , \sO_{\sZ}) .
}
\end{equation*}
We have just shown that  the left term is zero. Furthermore, the right term,  is zero as well, because
\begin{equation*}
 \Hom_{\sZ} ( L \alpha^* \bL_{X} , \sO_{\sZ}) 
\cong 
\underbrace{\Hom_{X} (  \bL_{X} , R\alpha_* \sO_{\sZ})}_{\textrm{by adjunction}}
\cong
\underbrace{\Hom_{X} (  \bL_{X} , \sO_{X})}_{\textrm{$\alpha$ is a coarse moduli map}}
=
\underbrace{0}_{ \parbox{70pt}{\tiny $X$ has no infinitesimal automorphisms}} .
\end{equation*}
This finishes our proof.

\end{proof}

\begin{lem} 
\label{lem:unique_extension}
Let $(\sX \to \sY \to \Spec A)$ be a fibration of stable stacks over an Artinian local algebra $A$ over $k$ such that,
\begin{equation*}
\Hom_{\sY}(\Omega_{\sY}, R^1 (\tg)_* \sO_{\sX})=0 . 
\end{equation*}
Let $A'$ be a small extension of $A$ and $\iota  \colon \sX \hookrightarrow \sX'$ a flat extension over $A'$. Then there is a unique (up to isomorphism)  extension $j \colon \sY \hookrightarrow \sY'$ and  an $A'$ morphism $\tg' \colon \sX' \to \sY'$, for which $\tg' \circ \iota = j \circ \tg$. 
\end{lem}

\begin{proof}
First, note that $\sX$ has no infinitesimal automorphisms by Lemma \ref{lem:infinitesmial_automorphisms}. Hence, \cite[Proposition 3.10]{Bhatt_Ho_Patakfalvi_Schnell_Moduli_of_products_of_stable_varieties} implies the unique existence of $\tg'$ and $\sY'$. 
\end{proof}

\begin{prop}
\label{prop:equivalence}
Let  $(\sX \to \sY \to \Spec k)$ be a fibration of stable stacks. Then, the natural forgetful map $\phi \colon \Def(\tg \colon \sX \to \sY) \to \Def(\sX)$ is an equivalence if 
$\Hom_{\sY}(\Omega_{\sY}, R^1 \tg_* \sO_{\sX})=0$. 
\end{prop}

\begin{proof}
Denote by  $\Art_{k, \leq l}$ and $\Art_{k, l}$ the category of Artinian local $k$-algebras $A$, such that $\dim_k A \leq l$ or $\dim_k A = l$, respectively. We prove by induction on $l$, that $\phi|_{\Art_{k,\leq l}}$ is an equivalence. The claim is vacuous for $l=1$. Hence we may assume that it is known for $l$ replaced by $l-1$. Choose any $A' \in \Art_{k,l}$. We may find a $A \in \Art_{k,l-1}$, such that $A'$ is a small extension of $A$. Choose now any $\sX' \in \Def(\sX)(A')$. We have to prove that there is a unique isomorphism class of $\Def(\sX \to \sY)$ mapping to $\sX'$. However, by our inductional hypothesis, this is known already for $\sX'_A \in \Def(\sX)(A)$. Then, Lemma \ref{lem:unique_extension} concludes our proof.
\end{proof}

\begin{prop}
\label{prop:conclusion}
The statement of Theorem \ref{thm:main} holds, i.e., the forgetful morphism $F \colon \fF \fM_{(h_1,h_2),h} \to \fM_{h}$ is \'etale, if $\Hom_{\sY}(\Omega_{\sY}, R^1 \tg_* \sO_{\sX})=0$ for every admissible fibration $( \sX \to \sY \to \Spec k)$ of stable stacks.
\end{prop}

\begin{proof}
Let $\uX=(X \to Y \to \Spec k)$ be a fibration of stable varieties as in \eqref{eq:fibration_of_stable_schemes}, and let $\usX=(\sX \to \sY \to \Spec k)$ be its index-one cover as in Definition \ref{defn:index_one_tower}. Further let $\widetilde{\sX}$ be the index-one cover of $X$ as in Definition \ref{defn:index_one_stack}. We are supposed to prove that the right vertical arrow of the following commutative diagram is an equivalence. However under the assumptions of the proposition all other arrows are equivalences, hence so is the right vertical arrow. 
\begin{equation*}
\xymatrix{
\Def ( \usX) \ar[dd]_{ \parbox{40pt}{\tiny forgetting the middle level}}^{\hspace{5pt}  \parbox{55pt}{\tiny equivalence by Proposition \ref{prop:equivalence}}} \ar[rrrrrrr]^{\textrm{taking coarse moduli space}}_{\textrm{equivalence by Theorem \ref{thm:tower_equivalence} and Lemma \ref{lem:admissible_open}}} & & & & & & &  \Def_{\bQ}( \uX)  \ar[dd]_{ \parbox{40pt}{\tiny forgetting the middle level}} & 
\\
\\
\Def (\sX) \ar@/_4pc/[rrrrrrr]^{\textrm{taking coarse moduli space}} &&  \Def(  \sX \to \widetilde{\sX} ) 
\ar[ll]_>>{\hspace{50pt}\textrm{equivalence by}}^>>{\hspace{50pt} \textrm{Proposition \ref{prop:two_index_ones}}} 
\ar[rr]^{\quad\textrm{equivalence by}}_{\quad \textrm{Proposition \ref{prop:two_index_ones}}} 
&  &  \Def(\widetilde{\sX}) \ar[rrr]^{ \textrm{taking coarse moduli space}}_{\textrm{equivalence by Theorem \ref{thm:abramovich_hassett}}} & &  &   \Def_{\bQ}( X ) 
}
\end{equation*}

\end{proof}

\section{Vanishing and negativity}
\label{sec:vanishing}

Disregarding issues about passing to index-one covers,  by Proposition \ref{prop:conclusion} we need to show a vanishing of $\Hom_Y(\Omega_{Y}, R^1 f_* \sO_{X})=0$ for families of stable varieties $f \colon X \to Y$ over stable bases.  By \cite{Kollar_Kovacs_Log_canonical_singularities_are_Du_Bois,Dubois_Jarraud_Une_propriete_de_commutation_au_changement_de_base_des_images_directes_superieures_du_faisceau_structural}, $R^1 f_* \sO_X$ is known to be a vector bundle. Hence, our approach is to show in this section that $R^1 f_* \sO_{X}$ is anti-nef, and then that $\Hom_Y(\Omega_{Y}, \sE)=0$ for every anti-nef vector bundle $\sE$. Recall that a vector bundle $\sE$ is anti-nef, if its dual $\sE^*$ is nef. For a general reference on ample vector bundles see \cite{Lazarsfeld_Positivity_in_algebraic_geometry_II}.  First the negativity statement:

\begin{prop}
\label{thm:semi_negative}
If $f\colon X \to Y$ is a flat, projective family of connected slc schemes with $\omega_{X/Y}$ relatively ample (as a $\bQ$-line bundle),  then $R^{1} f_* \sO_X$ is an anti-nef vector bundle or equivalently $R^{-1} f_* \omega_{X/Y}^{\bullet}$ is a nef vector bundle.  
\end{prop}

\begin{proof}
By \cite[Theorem 7.8]{Kollar_Kovacs_Log_canonical_singularities_are_Du_Bois}, $R^i f_* \sO_X$ is locally free and compatible with base change. Hence, we may assume that $Y$ is a smooth projective curve over $k$. We prove the statement by induction on $\dim X$. If $\dim X=2$, then by \cite[Theorem 21]{Kleiman_Relative_duality_for_quasicoherent_sheaves}, using that $X$ is Cohen-Macaulay by the dimension assumption, $(R^1 f_* \sO_X)^* \cong f_* \omega_{X/Y}$. However, the latter is nef by \cite[Theorem 4.12]{Kollar_Projectivity_of_complete_moduli}. 

If $\dim X >2$, then choose an ample enough hyperplane section $H$ of $X$. Let $g: H \to Y$ be the induced morphism. Since every fiber of $f$ is $S_2$, $\omega_{X_y}^{\bullet}$ is supported in cohomological degrees smaller than $-1$ for every $y \in Y$ \cite[Proposition 3.3.6]{Patakfalvi_Base_change_behavior_of_the_relative_canonical_sheaf}. Hence for a fixed $y \in Y$, $H^{-1}(X_y, (\omega_{X_y}^{\bullet}(H_y))=0$ by Serre-vanishing.  However then by Grothendieck duality, 
\begin{equation*}
H^1(X_y, \sO_{X_y}(-H_y)) = H^{-1}(X_y, (\omega_{X_y}^{\bullet}(H_y))^* = 0 .
\end{equation*}
Now, using flatness of $f$ and the semicontinuity of $\dim_{k(y)} H^1(X_y, \sO_{X_y}(-H_y))$, we obtain that $H^1(X_y, \sO_{X_y}(-H_y)) =0$ for any $y \in U$ where $U$ is a non-empty open set of $Y$. However, then replacing $H$ by an adequate power of itself, we obtain this vanishing also for the finitely many points of $Y \setminus U$. In particular then by cohomology and base change $R^1 f_* \sO_X(-H)=0$. 

Consider then the exact sequence
\begin{equation*}
\xymatrix{
0= R^1 f_* \sO_X(-H) \ar[r] & R^1 f_* \sO_X \ar[r] & R^1 g_* \sO_H  .
}
\end{equation*}
Since $H$ was general, $g$ is also a flat, projective family of connected slc schemes with $\omega_{H/Y}$ relatively ample. Hence by induction $R^1 g_* \sO_H$ is an anti-nef vector bundle. Then by the above exact sequence it follows that so is $R^1 f_* \sO_X$.

\end{proof}

Second, we prove  Theorem \ref{thm:vanishing_hom_from_cotangent}. The proof consists of two main parts. First, in Theorem \ref{thm:vanishing_log_canonical}, we show a generalization of a special case of Bogomolov Sommese vanishing for log-canonical spaces \cite[Theorem 7.2]{Greb_Kebekus_Kovacs_Peternell_Differential_forms_on_log_canonical_spaces}. In particular, Theorem \ref{thm:vanishing_log_canonical} implies Theorem \ref{thm:vanishing_hom_from_cotangent} when $X$ is irreducible. The second ingredient is Lemma \ref{lem:vector_fields} that allows us to conclude the reducible case using Theorem \ref{thm:vanishing_log_canonical}. Theorem \ref{thm:vanishing_log_canonical} uses the notation of reflexive tensor products (i.e., $[ \otimes ]$), reflexive differentials and (reflexive) $\bQ$-line bundles. We refer to Section \ref{subsec:notations} for the precise definitions.

Recall that the statement of Theorem \ref{thm:vanishing_log_canonical} is:

\begin{thm_vanishing_log_canonical}
If $X$ is a  projective variety of dimension $n$, $D \geq 0$ a $\bQ$-divisor on $X$
such that $(X,D)$ is log canonical, $\sL$ an anti-ample $\bQ$-line bundle, $\sE$ an
anti nef vector bundle, then
\begin{equation*}
 H^0(X, \Omega^{[n-1]}_X (\log \lfloor D \rfloor )  [\otimes] \sL \otimes
\sE ) = 0 .
\end{equation*}
\end{thm_vanishing_log_canonical}

%


\begin{proof}[Proof of Theorem \ref{thm:vanishing_log_canonical}]
First, we show that we may assume that $\sL$ is a line bundle. Choose an integer
$N>0$, so that $\sL^{[-N]}$ is a very ample line bundle, and a general section $s
\in \sL^{[-N]}$. Let $\tau \colon X' \to X$ be the $N$-degree cyclic cover of $X$
given by $\sL^{*}$ and $s$. In other words
\begin{equation*}
X' := \Spec_X \left( \bigoplus_{i=0}^{N-1} \sL^{[i]} \right),
\end{equation*}
where the algebra structure is given by the natural tensor operations and the
section $s$. Define $D':= \tau^*(D)$. Note that $\tau$ is ramified over an irreducible divisor $B$ determined by $s$, which avoids the general point of any component of $D$. Hence, by  \cite[Lemma 5.17.2 and Proposition 5.20]{Kollar_Mori_Birational_geometry_of_algebraic_varieties}, $(X',D')$ is log canonical. Furthermore $\sL'=\tau^{[*]} \sL$ is a line bundle.
If we knew the statement of the theorem for $\sL$ being a line bundle, then we
would have
\begin{equation}
\label{eq:vanishing_log_canonical:if_it_was_a_line_bundle}
H^0(X', \Omega^{[n-1]}_{X'} (\log \lfloor D' \rfloor )  \otimes \sL'
\otimes \tau^* \sE ) = 0 
\end{equation}
Let $U \subseteq X$ be the open locus of $X$ where both $X$ and $D +B$ are smooth and define $U':=\tau^{-1}(U)$.
Note first that $\sL$ is a line bundle over $U$, second that $U'$ and
$D'|_{U'}$ are also smooth and third that $\codim_X X \setminus U \geq 2$.
That is, \eqref{eq:vanishing_log_canonical:if_it_was_a_line_bundle} would imply 
\begin{equation}
\label{eq:vanishing_log_canonical:if_it_was_a_line_bundle2}
0 = 
\underbrace{H^0(U', \Omega^{n-1}_{U'} (\log \lfloor D'
\rfloor )  \otimes (\tau|_{U'})^* (\sL \otimes  \sE)
)}_{\textrm{\cite[Proposition 1.11]{Hartshorne_Generalized_divisors_on_Gorenstein_schemes} and
\eqref{eq:vanishing_log_canonical:if_it_was_a_line_bundle}}} 
= 
\underbrace{H^0(U, ((\tau|_{U'})_* \Omega^{n-1}_{U'}  (\log
\lfloor D' \rfloor )) \otimes  \sL|_U \otimes  \sE|_U
)}_{\textrm{projection formula}} .
\end{equation}
Note at this point that since both $D|_U$ and $B|_U$ are smooth, by \cite[Lemma 3.16.a]{Esnault_Viehweg_Lectures_on_vanishing_theorems}
\begin{equation}
\label{eq:vanishing_log_canonical:pullback_log_cotangent}
(\tau|_{U'})^* \Omega_{U}^{n-1} (\log \lfloor D \rfloor + B ) \cong \Omega_{U'}^{n-1} (\log \lfloor D' \rfloor + \tau^*B ) .
\end{equation}
Hence,
\begin{equation}
\label{eq:vanishing_log_canonical:direct_sum}
(\tau|_{U'})_* \Omega_{U'}^{n-1} (\log \lfloor D' \rfloor + \tau^*B ) 
%
%
%
%
%
\cong
\bigoplus_{i=0}^{N-1} \Omega_{U}^{n-1} (\log \lfloor D \rfloor + B ) \otimes \sL|_U^{i}.
\end{equation}
The natural embedding $\Omega_{U'}^{n-1} (\log \lfloor D' \rfloor ) \hookrightarrow \Omega_{U'}^{n-1} (\log \lfloor D' \rfloor + \tau^*B )$ and \eqref{eq:vanishing_log_canonical:direct_sum} yields an embedding 
\begin{equation*}
\iota \colon (\tau|_{U'})_* \Omega_{U'}^{n-1} (\log \lfloor D' \rfloor ) 
\hookrightarrow
\bigoplus_{i=0}^{N-1} \Omega_{U}^{n-1} (\log \lfloor D \rfloor + B ) \otimes \sL|_U^{i}.
\end{equation*}
We claim that 
\begin{equation}
\label{eq:vanishing_log_canonical:image}
\im \iota = \Omega_{U}^{n-1} (\log \lfloor D \rfloor  ) \oplus \left( \bigoplus_{i=1}^{N-1} \Omega_{U}^{n-1} (\log \lfloor D \rfloor + B ) \otimes \sL|_U^{i} \right). 
\end{equation}
Indeed, \eqref{eq:vanishing_log_canonical:image} is a local question, so since $U \cap \Supp B \cap  \Supp \lfloor D \rfloor=\emptyset$ it is enough to prove it over $U \setminus \Supp B$ and $U \setminus \Supp \lfloor D \rfloor$ separately. That is, we may assume that either $B=0$ or $D=0$. In the former case \eqref{eq:vanishing_log_canonical:direct_sum} and in the latter \cite[Lemma 3.16.d]{Esnault_Viehweg_Lectures_on_vanishing_theorems} proves \eqref{eq:vanishing_log_canonical:image}. Therefore, $(\tau|_{U'})_* \Omega^{n-1}_{U'} (\log \lfloor
D' \rfloor )$ has a direct factor isomorphic to $\Omega^{n-1}_U (\log \lfloor D
\rfloor )$. Hence, \eqref{eq:vanishing_log_canonical:if_it_was_a_line_bundle2} implies  that 
\begin{equation*}
 0 = H^0(U,  \Omega^{n-1}_{U} (\log \lfloor D \rfloor ) \otimes  \sL|_U
\otimes  \sE|_U ) 
 =
\underbrace{H^0(X,  \Omega^{[n-1]}_{X} (\log \lfloor D \rfloor ) [\otimes] 
\sL \otimes  \sE )  }_{\textrm{\cite[Proposition 1.11]{Hartshorne_Generalized_divisors_on_Gorenstein_schemes}}} .
\end{equation*}
Therefore, we may assume indeed that $\sL$ is a line bundle. 

Choose now a log-resolution $\pi \colon Y \to X$ of $(X,D)$. Let $\widetilde{D}$ be
the biggest reduced divisor in $\pi^{-1}(\textrm{non-klt locus of $(X,D)$})$. 
Then
\begin{multline*}
 H^0(X, \Omega^{[n-1]}_X (\log \lfloor D \rfloor )  [\otimes] \sL \otimes
\sE ) 
%
%
%
  \cong 
\underbrace{H^0(X, \pi_* \Omega^{n-1}_Y (\log \widetilde{D} )  \otimes
\sL \otimes \sE )}_{\textrm{\cite[Theorem 1.5]{Greb_Kebekus_Kovacs_Peternell_Differential_forms_on_log_canonical_spaces}}}
\\ 
\cong \underbrace{H^0(Y, \Omega^{n-1}_Y (\log \widetilde{D} )  \otimes \pi^*
\sL \otimes \pi^* \sE )}_{\textrm{projection formula}}
%
%
%
%
\cong \underbrace{\Hom_Y( \Omega^{1}_Y (\log \widetilde{D} ),
\omega_Y(\widetilde{D})  \otimes \pi^* \sL \otimes \pi^* \sE
)}_{\textrm{\cite[Exercise II.5.1.b]{Hartshorne_Algebraic_geometry}}} 
\end{multline*}
\emph{Assume now that this group is not zero.} Then there is a non-zero homomorphism
\begin{equation*}
 \phi \colon \Omega^{1}_Y (\log \widetilde{D} ) \to  \omega_Y(\widetilde{D})  \otimes
\pi^* \sL \otimes \pi^* \sE
\end{equation*}
Define $r := \rk (\im \phi)$. Note that $1 \leq r \leq n$. Then
\begin{equation}
\label{eq:vanishing_log_canonical:non-zero}
0 \neq \Hom(\Omega^{r}_Y (\log \widetilde{D}), (\wedge^r (\im \phi))^{**}  ) 
\end{equation}
Define $\sK:=(\wedge^r (\im \phi))^{**} \otimes \omega_Y(\widetilde{D})^*
\otimes \pi^* \sL^*$, and note that since $Y$ is smooth and $\sK$ is reflexive of
rank one, then  $\sK$ is a line bundle \cite[Proposition 1.9]{Hartshorne_Stable_reflexive_sheaves}.
Also note that there is an induced homomorphism $\sK \to \pi^* \wedge^r \sE$,
which is an embedding generically, and hence globally as well since $Y$ is integral. In particular, then $\sK$ is the inverse of a  pseudo-effective line bundle (\cite[Lemma 1.4.1]{Viehweg_Weak_positivity}, using that a weakly positive line bundle is pseudo-effective).  Therefore,
\begin{equation}
\label{eq:vanishing_log_canonical:contradiction}
0 \neq 
\underbrace{\Hom(\Omega^{r}_Y (\log \widetilde{D}), \omega_Y(\widetilde{D})
\otimes \pi^* \sL \otimes \sK )
}_{\textrm{\eqref{eq:vanishing_log_canonical:non-zero} and the definition of
$\sK$}}
%
%
\cong 
\underbrace{H^0(Y, \Omega^{n-r}_Y (\log \widetilde{D}) \otimes \pi^* \sL
\otimes \sK ) }_{\textrm{\cite[Exercise II.5.1.b]{Hartshorne_Algebraic_geometry}}} .
\end{equation}
However, $\pi^* \sL \otimes \sK \cong (\pi^* \sL^* \otimes \sK^*)^*$, and then
it is the dual of a big line bundle tensored with a pseudo-effective line bundle.
Hence, in fact, it is the dual of a big line bundle. But then the last group in
\eqref{eq:vanishing_log_canonical:contradiction} is zero by the Bogomolov
vanishing theorem \cite[Corollary 6.9]{Esnault_Viehweg_Lectures_on_vanishing_theorems}. This is a contradiction. So, our
assumption was false, which concludes our proof.
\end{proof}

The following lemma helps to deduce the non-normal case of Theorem \ref{thm:vanishing_hom_from_cotangent} from Theorem \ref{thm:vanishing_log_canonical}. For the definition of demi-normal please consult Section \ref{subsec:notations}. 

\begin{lem}
\label{lem:vector_fields}
If $X$ is a quasi-projective, equidimensional, demi-normal scheme , and $\pi \colon \oX \to X$ is its normalization
with conductor divisor $D \subseteq X$ and $\oD:= \pi^{-1}(D)_{\red}$, then
there is an inclusion
\begin{equation*}
  \sT_X \hookrightarrow \pi_* \sT_{\oX}(- \log \oD) .
\end{equation*}
(Here $\sT_X:=\sHom_X(\Omega_X, \sO_X)$ and $\sT_{\oX}(- \log \oD):= \sHom_X \left(\Omega_{\oX}^{[1]}( \log \oD), \sO_{\oX} \right)$, where $\Omega_{\oX}^{[1]}( \log \oD)$ is the sheaf of reflexive log-differentials, i.e., the reflexive hull of the sheaf of log-differentials on the normal crossing locus of $(\oX, \oD)$ \cite[2.17]{Greb_Kebekus_Kovacs_Peternell_Differential_forms_on_log_canonical_spaces}.)
\end{lem}

\begin{proof}
\emph{Step 1: we may assume that $X$ contains only smooth and nodal points.} Let $U$ be the open set of $X$ containing the smooth and double normal crossing
points. Define $\bar{U}:= \pi^{-1}(U)$. Both $\sT_X$ and $\sT_{\oX}(-
\log \oD)$ are reflexive, or equivalently $S_2$, by \cite[Corollary
1.8]{Hartshorne_Generalized_divisors_on_Gorenstein_schemes}. Then so is $\pi_* \sT_{\oX}(- \log \oD)$ by
\cite[Proposition 5.4]{Kollar_Mori_Birational_geometry_of_algebraic_varieties}. So, by \cite[Proposition 1.11]{Hartshorne_Generalized_divisors_on_Gorenstein_schemes} it is
enough to prove that there is a natural inclusion
\begin{equation*}
  \sT_U \hookrightarrow \pi_* \sT_{\bar{U}}(- \log \oD)   . 
\end{equation*}
With other words we may assume that $X$ contains only smooth and nodal points. 

\emph{Step 2: if $\sK(X)$ is the sheaf of total quotient rings, the kernel of $\Omega_X \to \Omega_X \otimes_{\sO_X} \sK(X)$ (given by $\eta \mapsto \eta \otimes 1$) is the submodule $\sC$ of sections the supports of which does not contain any component of $X$.} Since $\Omega_X$ is locally free at the generic points of the components, the kernel has to be contained in $\sC$. For the other containment, let $\eta$ be a local section of $\sC$, and $s$ a local section of $\sK(X)^\times \cap \sO_X$ such that $s \cdot \eta =0$. Then $\eta \otimes 1 = \eta \otimes (s \cdot s^{-1}) = s \eta \otimes s^{-1} = 0$. 

\emph{Step 3: $\sT_X \cong \sHom_X \left( \factor{\Omega_X}{\sC}, \sO_X \right)$.} This follows immediately from dualizing the exact sequence
\begin{equation*}
\xymatrix{
0 \ar[r] & \sC \ar[r] & \Omega_X \ar[r] & \factor{\Omega_X}{\sC} \ar[r] & 0
}
\end{equation*}
and noticing that $\sHom_X(\sC, \sO_X)=0$, since $X$ is $S_2$ and hence all the sections of $\sO_X$ are supported on the union of some components.

\emph{Step 4: it is enough to show that there is a natural inclusion $\pi_* \Omega_{\oX}\left(\log \oD \right) \left( - \oD \right) \hookrightarrow \factor{\Omega_X}{\sC}$, which is an isomorphism at the generic point of each component of $X$.} Indeed, by dualizing such an inclusion, we obtain an inclusion $\sT_X \hookrightarrow \sHom_X \left(\pi_* \Omega_{\oX}\left(\log \oD\right) \left( - \oD\right), \sO_X \right)$. Further,
\begin{multline*}
\sHom_X \left(\pi_* \Omega_{\oX}\left(\log \oD\right) \left( - \oD\right), \sO_X \right)
\cong 
\underbrace{\pi_* \sHom_X \left( \Omega_{\oX}\left(\log \oD\right) \left( - \oD\right), \omega_{\oX/X} \right)}_{\textrm{Grothendieck duality}}
\\ \cong 
\underbrace{\pi_* \sHom_X \left( \Omega_{\oX}\left(\log \oD\right) \left( - \oD\right), \sO_{\oX}\left(-\oD\right) \right)}_{\omega_{\oX/X} \cong \sO_{\oX}\left(-\oD\right)}
\cong 
\pi_* \sT_{\oX}\left(- \log \oD\right).
\end{multline*}

\emph{Step 5: showing an inclusion $\pi_* \Omega_{\oX}\left(\log \oD\right) \left( - \oD\right) \hookrightarrow \factor{\Omega_X}{\sC}$ as above.} Note that if $\iota_i : \xi_i \to X$ is the inclusion of the generic points of $X$, then
\begin{equation*}
\Omega_X \otimes_{\sO_X} \sK(X) \cong 
\underbrace{\bigoplus \iota_{i,*} \Omega_{X/k,\xi_i}}_{\textrm{\cite[Prop 2.1]{Hartshorne_Generalized_divisors_on_Gorenstein_schemes}, \cite[Prop 8.2.a]{Hartshorne_Algebraic_geometry}}}
 \cong \pi_* \left( \Omega_{\oX}(\log \oD) (-\oD) \otimes_{\sO_{\oX}} \sK(\oX) \right) .
\end{equation*}
So, in particular, both  $\pi_* \Omega_{\oX}\left(\log \oD\right) \left( - \oD\right)$ and $\factor{\Omega_X}{\sC}$ are both subsheaves of $\Omega_X \otimes_{\sO_X} \sK(X)$.  We verify that the first is a subsheaf of the second via these embeddings. Indeed, this is immediate at smooth points, because there they are equal. So, we may look at only the nodal points. Then after passing to an \'etale cover we may assume that we have a simple normal crossing point, that is,
\begin{equation*}
X=\Spec \frac{A[x,y]}{(xy)} \textrm{, where } A:= k[z_1,\dots,z_n] \textrm{, and }
\oX= \Spec \left( A[x] \oplus A[y] \right) .
\end{equation*}
Note that the embedding $\frac{A[x,y]}{(xy)} \hookrightarrow A[x] \oplus A[y]$ is the unique $A$-algebra homomorphism sending $x \mapsto x$ and $y \mapsto y$.
In this situation, $\sK(X)$ corresponds to the ring $A(x) \oplus A(y)$ viewed as an $\frac{A[x,y]}{(xy)}$-module, and $\Omega_X \otimes_{\sO_X} \sK(X)$ corresponds to the following  $A(x) \oplus A(y)$-module viewed as an $\frac{A[x,y]}{(xy)}$-module.
\begin{equation*}
B:=(A(x) dx \oplus A(y) dy) \oplus (A(x) \oplus A(y)) dz_1 \oplus \dots \oplus (A(x) \oplus A(y)) dz_n
\end{equation*}
Further, $\Omega_X$ corresponds to the $\frac{A[x,y]}{(xy)}$-module
\begin{equation*}
\frac{\frac{A[x,y]}{(xy)} dx \oplus \frac{A[x,y]}{(xy)} dy \oplus \frac{A[x,y]}{(xy)} dz_1 \oplus \dots \oplus \frac{A[x,y]}{(xy)} dz_n}{xdy + ydx},
\end{equation*}
where $\sC$ is the submodule generated by $xdy$. Consequently $\factor{\Omega_X}{\sC}$ corresponds to the following $\frac{A[x,y]}{(xy)}$-submodule of $B$. 
\begin{equation}
\label{eq:submodule_1}
A[x]  dx \oplus A[y] dy \oplus \frac{A[x,y]}{(xy)} dz_1 \oplus \dots \oplus \frac{A[x,y]}{(xy)} dz_n.
\end{equation}
 On the other hand, $\pi_* \Omega_{\oX}(\log \oD) ( - \oD)$ corresponds to the submodule
\begin{equation}
\label{eq:submodule_2}
A[x] dx \oplus A[y] dy \oplus (xA[x] \oplus yA[y]) dz_1 \oplus \dots \oplus (xA[x] \oplus yA[y]) dz_n.
\end{equation}
Since $(xA[x] \oplus yA[y])$ is a subring of $\frac{A[x,y]}{(xy)}$ when the latter is viewed embedded into $A[x] \oplus A[y]$,  submodule \eqref{eq:submodule_2} is indeed contained in submodule \eqref{eq:submodule_1}.
\end{proof}

\begin{proof}[Proof of Theorem  \ref{thm:vanishing_hom_from_cotangent}]
First, \emph{we claim that it is enough to show that $\Hom_X(\Omega_X, \sE)=0$}. Indeed, there is an exact triangle
\begin{equation*}
\xymatrix{
 \bL_X^{\leq -1} \ar[r] & \bL_X \ar[r] & \Omega_X \ar[r]^{+1} &  .
}
\end{equation*}
Hence applying
$\Hom(\_, \sE)$ gives the exact sequence
\begin{equation*}
\xymatrix{
\Hom(\Omega_X, \sE) \ar[r] & \Hom(\bL_X, \sE) \ar[r] & \Hom(\bL_X^{\leq -1}, \sE),
}
\end{equation*}
where the last term is zero, since $\bL_X^{\leq -1}$ is supported in negative cohomological degrees, while $\sE$ in zero cohomological degrees (recall that $\Hom (\_, \sE)$ is computed by $h^0(\Hom^{\bullet}(\_,\sI))$, where $\sI$ is an injective resolution  of $ \sE$, and then since $\bL_X^{\leq -1}$ is supported in negative cohomological degrees, $\Hom^{\bullet}(\bL_X^{\leq -1},\sI)=0$ holds).
This concludes our claim.

Now we show that $\Hom_X(\Omega_X, \sE)=0$. Let $\pi \colon \oX \to X$ be the normalization of $X$ with conductor divisor $D
\subseteq X$ and $\oD:= \pi^{-1}(D)_{\red}$. Then there is an inclusion 
\begin{multline*}
 \Hom_X(\Omega_X, \sE) 
\cong \underbrace{ H^0(X,\sT_X \otimes \sE)}_{\textrm{$\sE$ is locally free}}
\hookrightarrow \underbrace{ H^0(X,\pi_* \sT_{\oX}(- \log \oD) \otimes
\sE)}_{\textrm{Lemma \ref{lem:vector_fields}}} 
 \cong \\ \underbrace{ H^0(\oX,\sT_{\oX}(- \log \oD) \otimes \pi^*
\sE) }_{\textrm{projection formula}} 
 \cong \underbrace{ H^0(\oX,\Omega_{\oX}^{[n-1]}( \log \oD)
[\otimes] \omega_{\oX}(\oD)^* \otimes \pi^* \sE) }_{\textrm{wedge
pairing isomorphism}} 
.
\end{multline*}
Hence it is enough to prove that the last group is zero. However, that follows
from  Theorem \ref{thm:vanishing_log_canonical} by setting 
$\sL:=\omega_{\oX}(\oD)^*$, which is anti-ample by \cite[(5.7.1)]{Kollar_Singularities_of_the_minimal_model_program}. 
\end{proof}

\section{Proof of the main theorem}
\label{sec:main_theorem}

In this section we prove Theorem \ref{thm:main}.

\begin{lem}
\label{lem:vanishing_from_schemes_to_stacks}
Given a fibration $(X \to Y \to B)$ of stable varieties as in \eqref{eq:fibration_of_stable_schemes}, and its corresponding index-one fibration $(\sX \to \sY \to B)$ of stable stacks as in Definition \ref{defn:index_one_tower}, $\Hom_{\sY} (\Omega_{\sY},R^1 \tg_* \sO_{\sX})=0$.
\end{lem}

\begin{proof}
By Lemma \ref{lem:pi_is_a_coarse_moduli map}, the coarse moduli tower of  $(\sX \to \sY \to B)$   is $(X \to Y \to B)$. So, we use the notations of \eqref{eq:tower_of_stable_stacks_coarse_maps}, which we recall here:
\begin{equation*}
\begin{split}
\xymatrix{
    \sX =  \sX_2 \ar@/^2pc/[rr]^{\tilde{f}}  \ar[r]^{\tg=\tf_{2}} \ar[d]^{\gamma} & \sY= \sX_{1} \ar[r]^{\tf_{1}} \ar[d]^{\pi} &
 \sX_0 = B \ar@{=}[d] \\
     X= X_2 \ar@/_2pc/[rr]^{f}  \ar[r]^{g=f_{2}} & Y= X_{1} \ar[r]^{f_1} &
 X_0 = B
} ,
\end{split}
\end{equation*} 
By Proposition \ref{thm:semi_negative}, $R^1 g_* \sO_{X}$ is a weakly negative vector bundle. Then by Theorem \ref{thm:vanishing_hom_from_cotangent}, $\Hom_{Y}(\Omega_{Y}, R^1 g_* \sO_{X})=0$. However,  
\begin{equation}
\label{eq:vanishing_from_schemes_to_stacks:structure_sheaves}
R^1 g_*  \sO_{X}
\cong
\underbrace{R^1 g_* \gamma_* \sO_{\sX}}_{\parbox{85pt}{\tiny $\gamma_* \sO_{\sX} \cong \sO_{X}$, since $\gamma$ is a coarse moduli map}} 
\cong 
\underbrace{\pi_* R^1 \tg_*  \sO_{\sX}}_{\parbox{95pt}{\tiny $\pi_*$ is exact, since $\pi$ is a coarse moduli map}} ,
\end{equation}
and hence
\begin{equation}
\label{eq:vanishing_from_schemes_to_stacks:vanishing_pullback}
0 =  
\underbrace{\Hom_{Y}(\Omega_{Y}, \pi_* R^1 \tg_*  \sO_{\sX})}_{\textrm{by \eqref{eq:vanishing_from_schemes_to_stacks:structure_sheaves}}}
\cong
\underbrace{\Hom_{\sY}(L \pi^* \Omega_{Y},  R^1 \tg_*  \sO_{\sX})}_{\textrm{by adjunction}} 
\\ \cong 
\underbrace{\Hom_{\sY}(\pi^* \Omega_{Y},  R^1 \tg_*  \sO_{\sX})}_{\textrm{by cohomological degrees}} 
.
\end{equation}
Consider now the triangle 
\begin{equation}
\label{eq:vanishing_from_schemes_to_stacks:Omega_exact_triangle}
\xymatrix{
\pi^* \Omega_{Y} \ar[r] & \Omega_{\sY} \ar[r] & \Omega_{\sY/Y} \ar[r]^{+1} & .
}
\end{equation}
By \eqref{eq:vanishing_from_schemes_to_stacks:vanishing_pullback} and \eqref{eq:vanishing_from_schemes_to_stacks:Omega_exact_triangle}, it is enough to prove that $\Hom_{\sY}(\Omega_{\sY/Y},R^1 \tg_* \sO_{\sX})=0$. 
Since $\sY \to Y$ is isomorphism in codimension one, $\Omega_{\sY/Y}$ is supported on a codimension two closed set. Hence it is enough to prove that $R^1 \tg_* \sO_{\sX}$ is locally free. At this point, we are going to use the notations of Remark \ref{rem:coarse_tower_with_proj}. By Remark \ref{rem:coarse_tower_with_proj}, $R \rho_* \sO_{\sX} \cong \sO_{\sX'}$, where $\sX':=X \times_{Y} \sY$ and $\rho: \sX \to \sX'$ is the induced morphism. Denote by $g'$ the natural morphism $\sX' \to \sY$. Then 
\begin{equation*}
R^1 \tg_* \sO_{\sX} \cong R^1 (g' \circ \rho)_* \sO_{\sX} \cong h^1(R g'_*  R \rho_* \sO_{\sX}) \cong h^1(R g'_*  \sO_{\sX'}) \cong R^1 g'_* \sO_{\sX'}
\end{equation*}
However $g'$ is a family of stable schemes, so $R^1 g'_* \sO_{\sX'}$ is locally free by Theorem \cite[Theorem 7.8]{Kollar_Kovacs_Log_canonical_singularities_are_Du_Bois} (The base of $\sX' \to \sY$ is a DM-stack, so one has to be slightly careful when applying  \cite[Theorem 7.8]{Kollar_Kovacs_Log_canonical_singularities_are_Du_Bois}. Note that it is enough to prove that the  pullback of $R^1 g'_* \sO_{\sX'}$ to an \'etale cover $\zeta \colon Z \to \sY$ of $\sY$ by a scheme is locally free (in \'etale topology which follows from showing it in Zariski topology). However, $\zeta^*  R^1 g'_* \sO_{X'} \cong R^1 g_{Z,*} \sO_{X_Z}$, so over $Z$ \cite[Theorem 7.8]{Kollar_Kovacs_Log_canonical_singularities_are_Du_Bois} applies directly.)

\end{proof}

\begin{proof}[Proof of point \eqref{itm:vague:non_Q_Gorenstein} of Theorem \ref{thm:vague}]
 It follows from \cite[Propositions 3.9 and 3.10]{Bhatt_Ho_Patakfalvi_Schnell_Moduli_of_products_of_stable_varieties}, Theorem \ref{thm:vanishing_hom_from_cotangent} and Proposition \ref{thm:semi_negative}.
\end{proof}

\begin{proof}[Proof of Theorem \ref{thm:main} and equivalently of point \eqref{itm:vague:Q_Gorenstein} of Theorem \ref{thm:vague}]
It follows from Lemma \ref{lem:vanishing_from_schemes_to_stacks} and Proposition \ref{prop:conclusion}.
\end{proof}

\begin{rem}
Let us note that iterating the results of the paper one can obtain similar results to towers. We word these precisely here. Let a \emph{tower of stable varieties with Hilbert function vector $\uh=(h_1,\dots,h_n)$} over a base scheme $B$   ve a commutative diagram
\begin{equation}
\label{eq:tower_of_stable_schemes}
\xymatrix{
    X= X_n \ar@/^2pc/[rrrr]^f  \ar[r]_{f_{n}} & X_{n-1} \ar[r]_{f_{n-1}} & \dots 
\ar[r]_{f_{2}} & X_1 \ar[r]_{f_{1}} & X_0 = B
} ,
\end{equation}
such that $f_i$ is a family of stable varieties (satisfying Koll\'ar's condition), and $\chi \left(  \omega_{(X_i)_y}^{[m]} \right) = h_i(m)$ for every $m \in \bZ$, $1 \leq i \leq n$ and $y \in X_{i-1}$. Define the category fibered in groupoids $\fT \fM_{\uh}$ over $\Sch_k$ to have such towers as objects over $B$, and natural Cartesian pullbacks as morphisms. For a vector of integers $\um=(m_1,\dots, m_n)$ define also the category  of all towers with dimension vector $\um$ as follows.
\begin{equation*}
\fT \fM_{\um}:= \bigcup_{\uh=(h_1,\dots,h_n), \deg h_i = m_i} \fT \fM_{\uh}
\end{equation*}
By induction on $n$, $\fT \fM_{\um}$ is a DM-stack locally of finte type over $k$ (c.f. Proposition \ref{prop:stack_isomorphism}). Let $F \colon \fT \fM_{\um} \to \ofM_m$ denote the forgetful functor obtained by disregarding the middle levels of a tower (here $m= \sum m_i$). Then iterated use of Theorem \ref{thm:vague} yields that the forgetful functor $F \colon \fT \fM_{\um} \to \ofM_m$ is \'etale.

\end{rem}

\bibliographystyle{skalpha}
\bibliography{includeNice}

\end{document}